\documentclass[12pt]{article}
\usepackage{amssymb}
\usepackage{amsmath}
\usepackage{amsthm}
\usepackage{color}
\usepackage{comment}
\usepackage{geometry}		% allows easy margin settings
\usepackage{graphicx}

\DeclareFontFamily{OT1}{pzc}{}
\DeclareFontShape{OT1}{pzc}{m}{it}{<-> s * [1.10] pzcmi7t}{}
\DeclareMathAlphabet{\mathpzc}{OT1}{pzc}{m}{it}

\let\originalleft\left
\let\originalright\right
\renewcommand{\left}{\mathopen{}\mathclose\bgroup\originalleft}
\renewcommand{\right}{\aftergroup\egroup\originalright}

\begin{document}

\def\b0{{\bf 0}}
\def\cO{\mathcal{O}}

\newcommand{\sect}[1]{\S\ref{sec:#1}}

\newtheorem{theorem}{Theorem}
\newtheorem{corollary}[theorem]{Corollary}
\newtheorem{lemma}[theorem]{Lemma}
\newtheorem{proposition}[theorem]{Proposition}

\title{
A general framework for boundary equilibrium bifurcations of Filippov systems.
}
\author{
D.J.W.~Simpson\\\\
Institute of Fundamental Sciences\\
Massey University\\
Palmerston North\\
New Zealand
}
\maketitle

% keywords:
% MSC codes:

\begin{abstract}

As parameters are varied a boundary equilibrium bifurcation (BEB) occurs when an equilibrium collides with a discontinuity surface in a piecewise-smooth system of ODEs. Under certain genericity conditions, at a BEB the equilibrium either transitions to a pseudo-equilibrium (on the discontinuity surface) or collides and annihilates with a coexisting pseudo-equilibrium. These two scenarios are distinguished by the sign of a certain inner product. Here it is shown that this sign can be determined from the number of unstable directions associated with the two equilibria by using techniques developed by Feigin. A new normal form is proposed for BEBs in systems of any number of dimensions. The normal form involves a companion matrix, as does the leading order sliding dynamics, and so the connection to the stability of the equilibria is explicit. In two dimensions the parameters of the normal form distinguish, in a simple way, the eight topologically distinct cases for the generic local dynamics at a BEB. A numerical exploration in three dimensions reveals that BEBs can create multiple attractors and chaotic attractors, and that the equilibrium at the BEB can be unstable even if both equilibria are stable. The developments presented here stem from seemingly unutilised similarities between BEBs in discontinuous systems (specifically Filippov systems as studied here) and BEBs in continuous systems for which analogous results are, to date, more advanced.

\end{abstract}

%===============================================================================
\section{Introduction}
\label{sec:intro}
%\setcounter{equation}{0}

%The phase space of a Filippov system is divided into regions $\Omega_i$ within each of which
%evolution is governed by a smooth set of ODEs, $\dot{x} = F^i(x)$.
The phase space of a Filippov system is divided into regions $\Omega_i$ within %each of
which evolution is governed by a smooth set of ODEs. % \cite{Fi88}.
Boundaries between regions, termed {\em discontinuity surfaces},
are assumed to be codimension-one and smooth (or possibly piecewise-smooth).
When an orbit reaches a discontinuity surface
(as we follow it with increasing time),
there are two generic possibilities for its subsequent motion: crossing and sliding.
%To be clear, suppose the orbit resides in $\Omega_1$ until reaching a discontinuity surface $\Sigma$ at a point $y$.
%When an orbit, in $\Omega_1$ say, reaches a discontinuity surface, call it $\Sigma$,
%there are two generic possibilities for its subsequent motion.
To be more precise, suppose the orbit resides in $\Omega_1$ until reaching a discontinuity surface $\Sigma$
bounding $\Omega_1$ and $\Omega_2$.
If the system in $\Omega_2$ points away from $\Sigma$, 
as in Fig.~\ref{fig:crossingSlidingDiscont}-A,
then the orbit {\em crosses} $\Sigma$ and enters $\Omega_2$.
Alternatively if the system in $\Omega_2$ points towards the discontinuity surface, 
as in Fig.~\ref{fig:crossingSlidingDiscont}-B,
then the orbit subsequently {\em slides} on $\Sigma$.
Such sliding motion is governed by a convex combination of the systems in $\Omega_1$ and $\Omega_2$.

%%%%%%%%%%%%%%%%%%%%%%%%%%%%%%%%%%%%%%%%%%%%%%%%%%%%%%%%%%%%%%%%%%%%%%%%%%%%%%%%%%%%%%%%%%%%%%%%%%%%%%%%%%%%%%%%%%%%%%%%
\begin{figure}[h!]
\begin{center}
\setlength{\unitlength}{1cm}
\begin{picture}(8,2.7)
\put(0,0){\includegraphics[width=3.6cm]{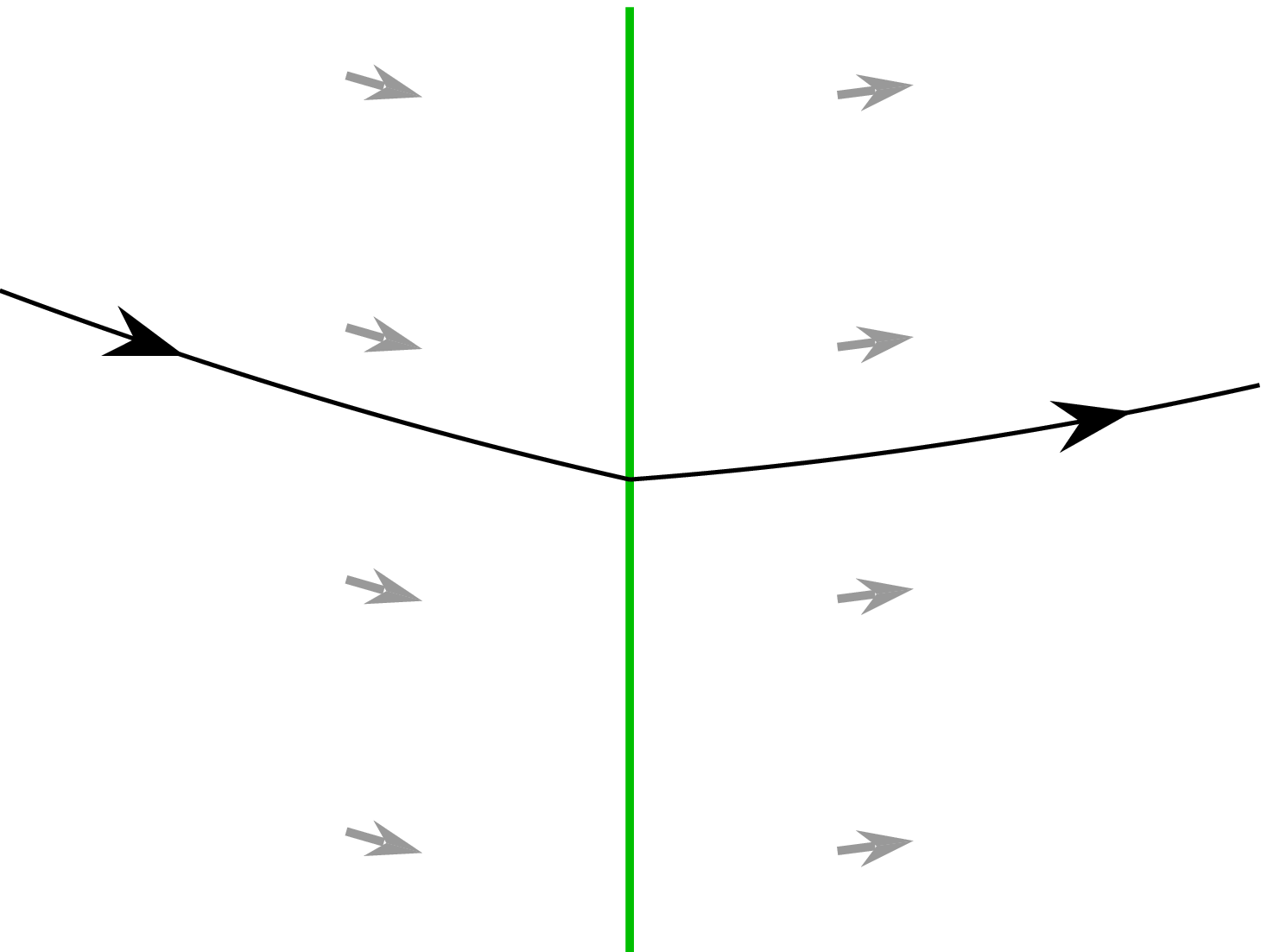}}
\put(4.4,0){\includegraphics[width=3.6cm]{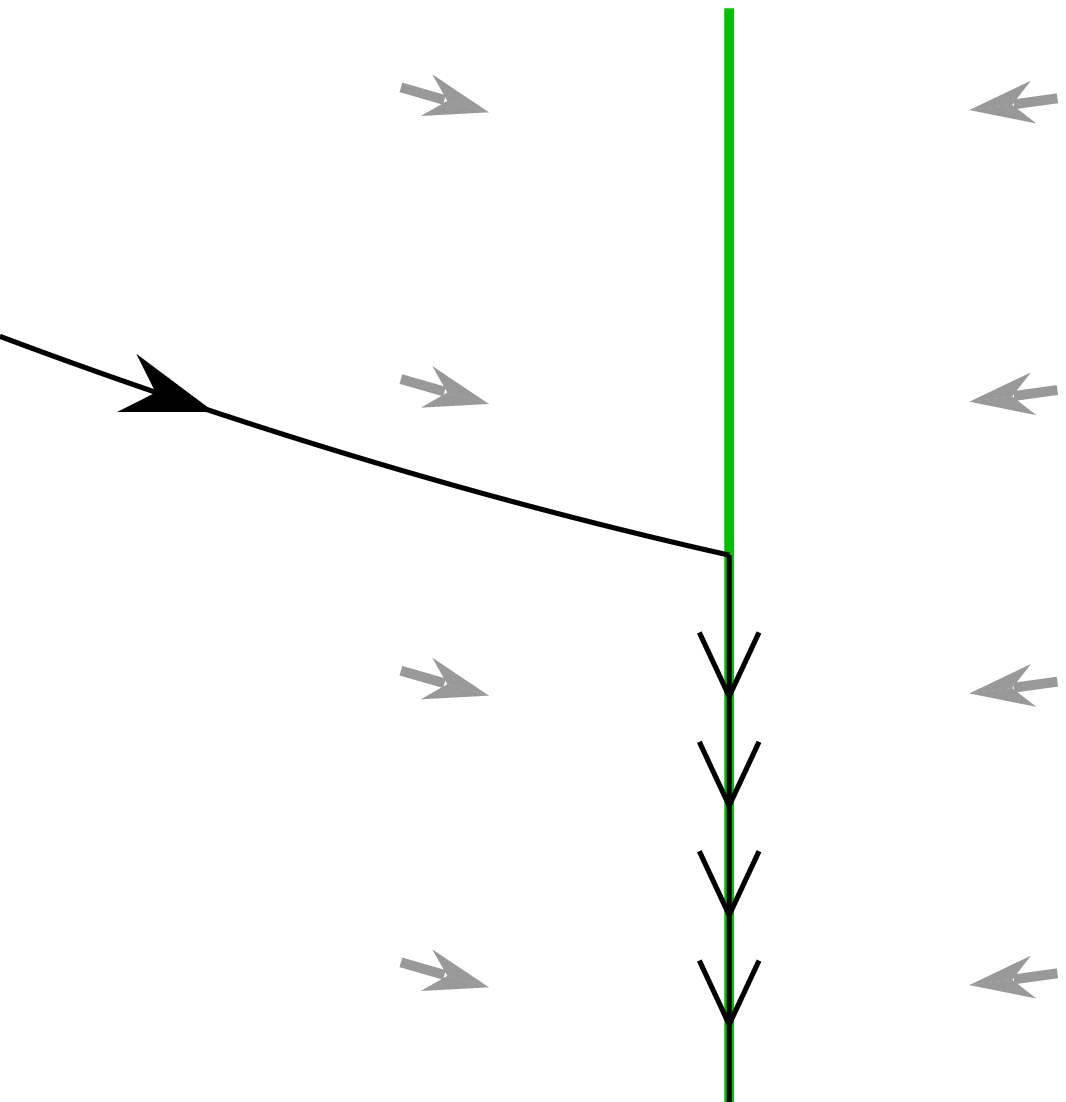}}
\put(0,2.45){\sf \bfseries A}
\put(0,.6){\footnotesize $\Omega_1$}
\put(3.25,.6){\footnotesize $\Omega_2$}
\put(1.85,0){\footnotesize $\Sigma$}
\put(4.4,2.45){\sf \bfseries B}
\put(4.4,.6){\footnotesize $\Omega_1$}
\put(7.65,.6){\footnotesize $\Omega_2$}
\put(6.25,0){\footnotesize $\Sigma$}
\end{picture}
\caption{
Sketches of an orbit of a Filippov system crossing a discontinuity surface (in panel A)
and sliding on the discontinuity surface (in panel B).
\label{fig:crossingSlidingDiscont}
} 
\end{center}
\end{figure}
%%%%%%%%%%%%%%%%%%%%%%%%%%%%%%%%%%%%%%%%%%%%%%%%%%%%%%%%%%%%%%%%%%%%%%%%%%%%%%%%%%%%%%%%%%%%%%%%%%%%%%%%%%%%%%%%%%%%%%%%

In mathematical models, sliding motion usually has important physical interpretations.
For instance in simple models of mechanical systems with stick-slip friction,
sliding motion corresponds to the sticking phase of motion \cite{BlCz99}.
For relay control systems, sliding motion represents the
idealised limit that the time between switching events is zero \cite{Jo03}.
In an ecological model of Dercole {\em et.~al.}~\cite{DeGr07},
sliding motion corresponds to predators that are hesitating between two sources of food.

%Mathematically, sliding motion is governed by a convex combination
%of the ODEs on each side of the discontinuity surface.

%As the parameters of a Filippov system are varied in a continuous manner,
As parameters are varied,
an equilibrium of a Filippov system can collide with a discontinuity surface.
Such BEBs have been identified in mathematical models of a wide variety of physical systems,
see for instance \cite{DeDe11,Kr11,TaLi12,WaWi16}.
Various invariant sets (such as limit cycles) can be created in BEBs.
But if we look only at equilibria then there are two generic scenarios.
These are distinguished by the relative coexistence of the equilibrium undergoing the BEB,
termed a {\em regular equilibrium},
and a {\em pseudo-equilibrium}: an equilibrium of the sliding dynamics.
If these two equilibria do not coexist, we effectively have the `persistence' of a single equilibrium.
If the equilibria do coexist, then they collide and annihilate in a `nonsmooth-fold'.
%These are `persistence', where the equilibrium turns into a {\em pseudo-equilibrium} (that is, an equilibrium of the sliding dynamics),
%and a `nonsmooth-fold' between the equilibrium and a coexisting pseudo-equilibrium.
%The latter case mimicks a saddle-node bifurcation.
%To classify a given BEB, as either persistence or a nonsmooth-fold,
To determine which situation occurs for a given BEB,
one can evaluate a certain inner product \cite{DiBu08,DiPa08}. % also [DiNo08]

%To attach deeper meaning
To develop this further, let us recall what is known about BEBs in
piecewise-smooth systems that are continuous but non-differentiable on discontinuity surfaces,
in this context termed switching manifolds.
For continuous systems, generic BEBs again conform to the two scenarios
of persistence and a nonsmooth-fold,
but here both equilibria are regular (one of each side of the switching manifold).
%Namely `persistence', where a regular equilibrium on one side of a switching manifold, $\Sigma$,
%turns into regular crosses from one side of the switching manifold to the other,
%and a `nonsmooth-fold', where equilibria on each side of the switching manifold collide and annihilate.
Locally, the stability of these equilibria is determined by the eigenvalues of the Jacobian matrices
of the two relevant smooth components of the system evaluated at the bifurcation.
%The assumption of continuity implies that these matrices differ by a rank-one matrix.
%This underpins the connection between the classification of BEB and the eigenvalues of these Jacobians,
%which can be stated as follows.
As was first shown for piecewise-smooth maps by Feigin \cite{Fe78} (for which the required calculations are almost identical),
the BEB corresponds to persistence
if the sum of the number of positive eigenvalues associated with each equilibrium is even,
and is a nonsmooth-fold if this sum is odd \cite{DiBu08}.

Here we show that a similar result holds for Filippov systems, where now one equilibrium is a pseudo-equilibrium.
%The discontinuous case has the additional complexity of sliding motion
%The difficulty lies dealing with the stability of the pseudo-equilibrium,
%but we find that again the Jacobian matrices of the equilibria differ by a rank-one matrix.
The key step in our derivation is to use
the matrix determinant lemma to connect the stability of the pseudo-equilibrium to the known inner product.
An immediate consequence is that if both equilibria are stable,
then neither has a positive eigenvalue %, hence the sum is zero
and so the BEB corresponds to persistence.

In order to understand other invariant sets created in BEBs,
it is in general not possible to employ dimension reduction techniques
that are invaluable for high-dimensional smooth systems of ODEs.
BEBs do not involve centre manifolds and so, as with maps \cite{GlJe15},
it appears that in $n$-dimensional systems BEBs can be inextricably $n$-dimensional \cite{Gl17b}.

BEBs are trivial in one dimension as equilibria are the only possible invariants.
BEBs in two dimensions were studied in detail by Kuznetsov {\em et.~al.}~\cite{KuRi03}
but have only recently been completely classified.
While it has long been known that there are eight topologically distinct cases for the
generic local dynamics of a system at a BEB \cite{Fi88},
some of these cases have multiple unfoldings and the realisation that there are exactly $12$
topologically distinct BEBs in two dimensions was first made by Hogan {\em et.~al.}~\cite{HoHo16}.

To facilitate studies of BEBs in more than two dimensions,
here we introduce an $n$-dimensional BEB normal form.
Normal forms given previously typically use a real Jordan form or a symmetric matrix
for the Jacobian matrix of the regular equilibrium \cite{DeTo14,Gl16d}.
Here a companion matrix is used because,
as with BEBs in continuous systems \cite{DiBu08,CaFr02},
such matrices are well-suited for coordinate transformations that leave the discontinuity surface unchanged.
As an added benefit, the Jacobian matrix of the pseudo-equilibrium is also a companion matrix.
Below we show that a Filippov system with a non-degenerate BEB can be transformed to the normal form
if and only if the Jacobian matrix of the regular equilibrium has no eigenvector tangent to the discontinuity surface.

The remainder of the paper is organised as follows.
We first formulate BEBs in a general setting and clarify sliding motion, \sect{setup}.
We then compute equilibria, \sect{eq},
and relate the relative coexistence of the equilibria to their associated eigenvalues, \sect{Feigin}.
Complete derivations are provided in \sect{proof} and consequences for codimension-two BEBs are discussed in \sect{codim2}.
The normal form is introduced in \sect{normalForm},
and its basic properties are discussed in \sect{properties}.
Section \ref{sec:2d} relates the normal form in two dimensions to known results,
and \sect{3d} provides a brief numerical exploration of the normal form in three dimensions.
Here we discover a chaotic attractor (similar to Shilnikov chaos described by Glendinning \cite{Gl17d}),
multiple attractors, and the lack of an attractor in a case for which all associated eigenvalues have negative real part.
Finally, conclusions are presented in \sect{conc}.

Throughout this paper,
$e_1,\ldots,e_n$ denote the standard basis vectors of $\mathbb{R}^n$,
and $\b0 \in \mathbb{R}^n$ denotes the zero vector (or origin).

%===============================================================================
\section{Preliminaries}
\label{sec:setup}

We consider systems of the form
\begin{equation}
\dot{x} = \begin{cases}
F^L(x;\mu), & x_1 < 0, \\
F^R(x;\mu), & x_1 > 0,
\end{cases}
\label{eq:F}
\end{equation}
where $F^L$ and $F^R$ have continuous second derivatives.
%\removableFootnote{
%I need to check that $C^2$ is sufficient.
%In particular, are error terms in some later expressions really $\cO \left( \mu^2 \right)$?
%}.
Here $x \in \mathbb{R}^n$ is the state variable and $\mu \in \mathbb{R}$ is a parameter.
The discontinuity surface, call it $\Sigma$, is where the first component of $x$ vanishes: $x_1 = 0$.
For systems with discontinuity surfaces that take a more general form, say $H(x) = 0$,
the idea is that one could apply a coordinate transformation to convert it to the form \eqref{eq:F}, at least locally.

Let
\begin{equation}
\chi(x;\mu) = F^L_1(x;\mu) F^R_1(x;\mu),
\label{eq:chi}
\end{equation}
be the product of the first components of $F^L$ and $F^R$.
%Following \cite{DiBu08},
Subsets of $\Sigma$ for which $\chi > 0$ are {\em crossing regions}
(in Fig.~\ref{fig:crossingSlidingDiscont}-A we have $F^L_1 > 0$ and $F^R_1 > 0$).
Subsets of $\Sigma$ for which $\chi < 0$ are {\em sliding regions}.
A sliding region is {\em attracting} if $F^L_1 > 0$ and $F^R_1 < 0$ (as in Fig.~\ref{fig:crossingSlidingDiscont}-B),
and {\em repelling} if $F^L_1 < 0$ and $F^R_1 > 0$.

Dynamics on a sliding region are governed by $\dot{x} = F^S(x;\mu)$
%\begin{equation}
%\dot{x} = F^S(x;\mu),
%\label{eq:Fslide1}
%\end{equation}
where
\begin{equation}
F^S = \frac{F^L_1 F^R - F^R_1 F^L}{F^L_1 - F^R_1},
\label{eq:Fslide2}
\end{equation}
is the unique convex combination of $F^L$ and $F_R$ for which $F^S_1 = 0$ \cite{DiBu08,Fi88}.

Now suppose $F^L$ has an equilibrium at $x = \b0$ when $\mu = 0$,
as illustrated in Fig.~\ref{fig:BEBFilippovSchem1_framework}.
Then
\begin{equation}
F^L(x;\mu) = A x + b \mu + \cO(2), 
\label{eq:FL}
\end{equation}
for some $n \times n$ matrix $A$ and $b \in \mathbb{R}^n$,
and we write $\cO(k)$ for terms that are order $k$ or greater in $x$ and $\mu$.
Notice that $A$ is the Jacobian matrix $D F^L(\b0;0)$.
Also
\begin{equation}
F^R(x;\mu) = c + \cO(1),
\label{eq:FR}
\end{equation}
for some $c \in \mathbb{R}^n$, and so our system has the form
\begin{equation}
\dot{x} = \begin{cases}
A x + b \mu + \cO(2), & x_1 < 0, \\
c + \cO(1), & x_1 > 0.
\end{cases}
\label{eq:F2}
\end{equation}
Locally, orbits in $x_1 > 0$ approach $\Sigma$ (as time increases) if $c_1 < 0$,
and head away from $\Sigma$ if $c_1 > 0$.
By substituting the above expressions for $F^L$ and $F^R$ into \eqref{eq:Fslide2}, we obtain
\begin{equation}
F^S(x;\mu) = \left( I - \frac{c e_1^{\sf T}}{c_1} \right)
(A x + b \mu) + \cO(2),
\label{eq:Fslide3}
\end{equation}
assuming $c_1 \ne 0$.

%%%%%%%%%%%%%%%%%%%%%%%%%%%%%%%%%%%%%%%%%%%%%%%%%%%%%%%%%%%%%%%%%%%%%%%%%%%%%%%%%%%%%%%%%%%%%%%%%%%%%%%%%%%%%%%%%%%%%%%%
\begin{figure}[h!]
\begin{center}
\setlength{\unitlength}{1cm}
\begin{picture}(5.2,3.9)
\put(0,0){\includegraphics[width=5.2cm]{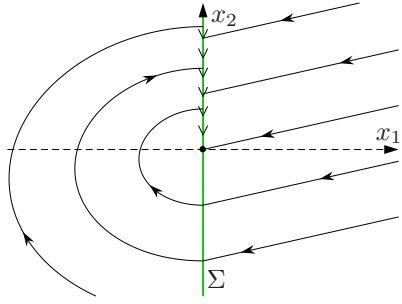}}
\put(4.9,2.08){\footnotesize $x_1$}
\put(2.71,3.65){\footnotesize $x_2$}
\put(2.65,.1){\footnotesize $\Sigma$}
\end{picture}
\caption{
A typical phase portrait of \eqref{eq:F2} with $\mu = 0$ (i.e.~at the BEB).
\label{fig:BEBFilippovSchem1_framework}
} 
\end{center}
\end{figure}
%%%%%%%%%%%%%%%%%%%%%%%%%%%%%%%%%%%%%%%%%%%%%%%%%%%%%%%%%%%%%%%%%%%%%%%%%%%%%%%%%%%%%%%%%%%%%%%%%%%%%%%%%%%%%%%%%%%%%%%%

%===============================================================================
\section{Equilibria}
\label{sec:eq}

A {\em regular equilibrium} of \eqref{eq:F} is a point $x \in \mathbb{R}^n$
for which $F^L(x;\mu) = \b0$ or $F^R(x;\mu) = \b0$.
For the system \eqref{eq:F2}, $F^R(x;\mu) = \b0$ has no local solution if $c \ne \b0$.
%Locally, there exists a unique regular equilibrium of the left-half system
%if and only if $\det(A) \ne 0$.
Solving $F^L(x;\mu) = \b0$ yields
\begin{equation}
x^L(\mu) = -A^{-1} b \mu + \cO \left( \mu^2 \right),
\label{eq:xL}
\end{equation}
assuming $\det(A) \ne 0$, and so we have the following result.

%...............................................................................
\begin{lemma}
If $\det(A) \ne 0$ and $c \ne \b0$, then, in a neighbourhood of $(x;\mu) = (\b0,0)$,
the system \eqref{eq:F2} has a unique regular equilibrium $x^L(\mu)$ given by \eqref{eq:xL}.
%\removableFootnote{
%The converse is not necessarily true because of the nonlinear terms in $F^L$.
%}.
\label{le:xL}
\end{lemma}

Since $F^L$ only applies to points with $x_1 < 0$,
we say that $x^L$ is {\em admissible} if $x^L_1 < 0$, and {\em virtual} if $x^L_1 > 0$.

A {\em pseudo-equilibrium} of \eqref{eq:F} is a point $x \in \Sigma$ for which $F^S(x;\mu) = \b0$.
%and $x$ belongs to a sliding region (i.e.~$\chi(x;\mu) < 0$).
%Uniqueness of the pseudo-equilibrium is more delicate.
To calculate pseudo-equilibria of \eqref{eq:F2}, first observe that the form \eqref{eq:Fslide3}
hides the fact that sliding dynamics is $(n-1)$-dimensional.
In \eqref{eq:Fslide3} we are assuming $x_1 = 0$; also $\dot{x}_1 = 0$.
For this reason, we let
%there exists a unique local pseudo-equilibrim, denoted $x^S(\mu)$,
%if $\det(\tilde{M}) \ne 0$,
$\tilde{M}$ denote the lower-right $(n-1) \times (n-1)$ block of
%\removableFootnote{
%From \eqref{eq:M} we can see that `degeneracies' (meaning what exactly?) in $A$
%get carried over to $\tilde{B}$.
%If $A$ is non-degenerate, as in the normal form below,
%then $\tilde{B}$ could be `anything'.
%}
\begin{equation}
M = D F^S(\b0;0) = \left( I - \frac{c e_1^{\sf T}}{c_1} \right) A.
\label{eq:M}
\end{equation}
The matrix $\tilde{M}$ is the Jacobian of the sliding dynamics evaluated at $(x;\mu) = (\b0,0)$,
and so we have the following result.

%...............................................................................
\begin{lemma}
If $\det(\tilde{M}) \ne 0$ and $c_1 \ne 0$, then, in a neighbourhood of $(x;\mu) = (\b0,0)$,
the system \eqref{eq:F2} has a unique pseudo-equilibrium $x^S(\mu)$.
\label{le:xS}
\end{lemma}

Notice $F^S(\b0;0) = \b0$, thus, by uniqueness, $x^S(0) = \b0$.
Since $F^S$ only applies to points in sliding regions of $\Sigma$,
we say that $x^S$ is {\em admissible} if
$\chi(x^S(\mu);\mu) < 0$, and {\em virtual} if $\chi(x^S(\mu);\mu) > 0$.

Now let us think about how the admissibility of $x^L$ and $x^S$ change as
the value of $\mu$ changes sign.
Since $x^L(0) = x^S(0) = \b0$, we can write 
\begin{align}
x^L_1(\mu) &= \alpha_L \mu + \cO \left( \mu^2 \right),
\label{eq:alphaL} \\
\chi(x^S(\mu);\mu) &= \alpha_S \mu + \cO \left( \mu^2 \right),
\label{eq:alphaS}
\end{align}
for some $\alpha_L, \alpha_S \in \mathbb{R}$.
Then $x^L$ is admissible if $\alpha_L \mu < 0$, and $x^S$ is admissible if $\alpha_S \mu < 0$.
If $x^L$ and $x^S$ are admissible for different signs of $\mu$,
the BEB is referred to as {\em persistence}, Fig.~\ref{fig:persistNSFoldDiscont}-A.
If $x^L$ and $x^S$ are admissible for the same sign of $\mu$,
the BEB is referred to as a {\em nonsmooth-fold}, Fig.~\ref{fig:persistNSFoldDiscont}-B.
Immediately we have the following result.

%%%%%%%%%%%%%%%%%%%%%%%%%%%%%%%%%%%%%%%%%%%%%%%%%%%%%%%%%%%%%%%%%%%%%%%%%%%%%%%%%%%%%%%%%%%%%%%%%%%%%%%%%%%%%%%%%%%%%%%%
\begin{figure}[b!]
\begin{center}
\setlength{\unitlength}{1cm}
\begin{picture}(8,2.7)
\put(0,0){\includegraphics[width=3.6cm]{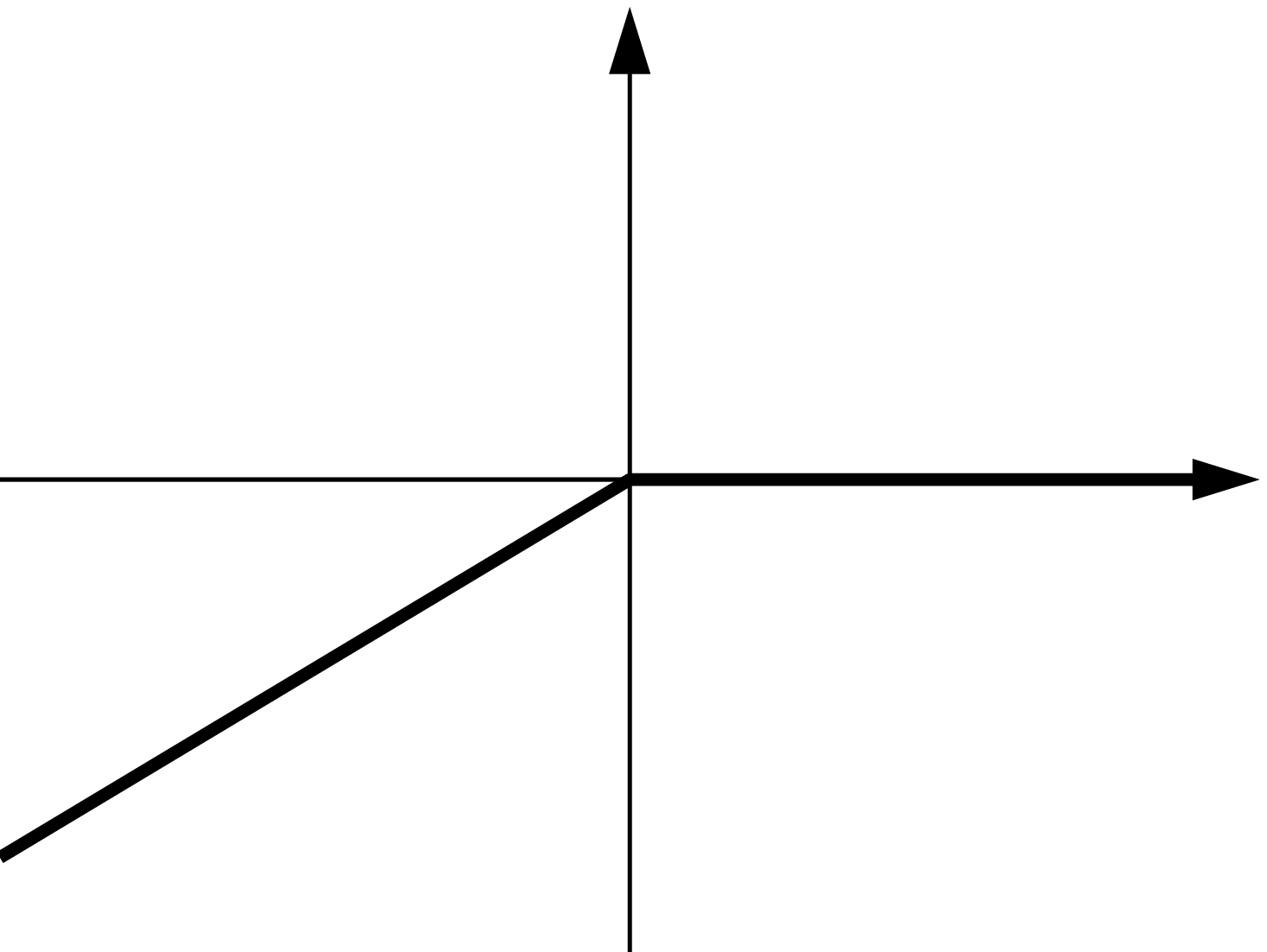}}
\put(4.4,0){\includegraphics[width=3.6cm]{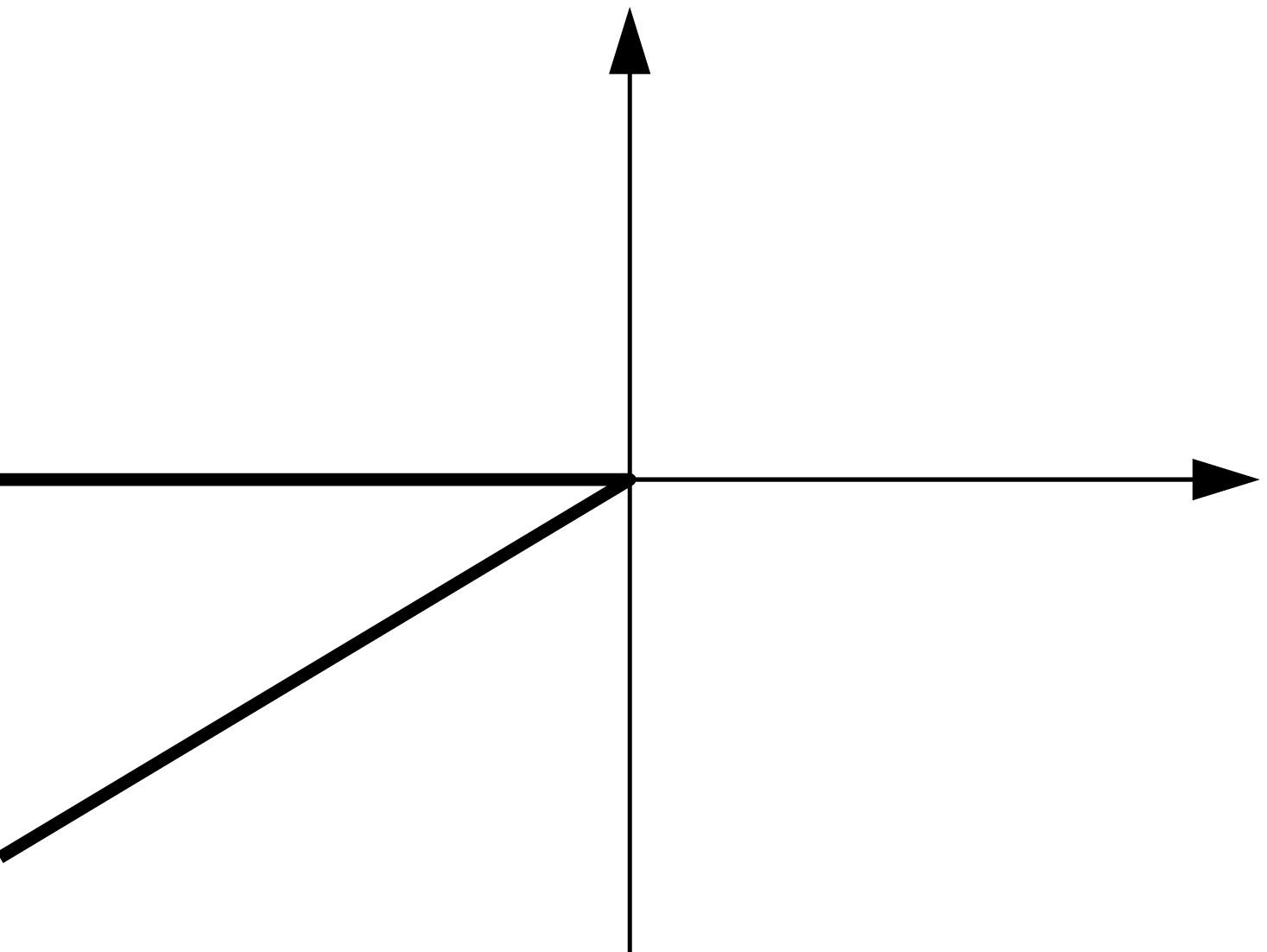}}
\put(0,2.45){\sf \bfseries A}
\put(3.2,1.46){\footnotesize $\mu$}
\put(1.91,2.45){\footnotesize $x_1$}
\put(.2,.54){\scriptsize $x^L$}
\put(2.2,1.42){\scriptsize $x^S$}
\put(4.4,2.45){\sf \bfseries B}
\put(7.6,1.46){\footnotesize $\mu$}
\put(6.31,2.45){\footnotesize $x_1$}
\put(4.6,.54){\scriptsize $x^L$}
\put(5,1.42){\scriptsize $x^S$}
\end{picture}
\caption{
Typical bifurcation diagrams of \eqref{eq:F2}
showing persistence (panel A) and a nonsmooth-fold (panel B).
\label{fig:persistNSFoldDiscont}
} 
\end{center}
\end{figure}
%%%%%%%%%%%%%%%%%%%%%%%%%%%%%%%%%%%%%%%%%%%%%%%%%%%%%%%%%%%%%%%%%%%%%%%%%%%%%%%%%%%%%%%%%%%%%%%%%%%%%%%%%%%%%%%%%%%%%%%%

%...............................................................................
\begin{lemma}
Suppose $\det(A) \ne 0$, $\det(\tilde{M}) \ne 0$, and $c_1 \ne 0$.
Then the BEB at $\mu = 0$ corresponds to persistence if $\alpha_L \alpha_S < 0$,
and to a nonsmooth-fold if $\alpha_L \alpha_S > 0$.
\label{le:classification}
\end{lemma}

%Clearly, for a given BEB, we need to evaluate $\alpha_L$ and $\alpha_S$,
%and this is the subject of the next two sections.
Computations of $\alpha_L$ and $\alpha_S$ form the subject of the next two sections.

%===============================================================================
%\section{A Feigin-type classification}
\section{Feigin analysis}
\label{sec:Feigin}

In a key 1978 paper \cite{Fe78},
Feigin showed that for border-collision bifurcations of piecewise-smooth continuous maps,
the existence and relative coexistence of fixed points and period-two solutions
can be determined, in a simple way, from the eigenvalues of the two corresponding Jacobian matrices.
For BEBs of piecewise-smooth continuous ODEs, the computations are almost identical \cite{DiBu08}.
The following result (proved in \sect{proof}) shows how this `Feigin analysis' extends to BEBs of \eqref{eq:F2}.
We assume $\alpha_L \ne 0$ to ensure that $\mu$ unfolds the bifurcation in a generic fashion
(see also the comments at the start of \sect{proof}).
For any $a \in \mathbb{R}$, we write
%${\rm sgn}(a) = -1$ if $a < 0$, and ${\rm sgn}(a) = 1$ if $a > 0$.
\begin{equation}
{\rm sgn}(a) = \begin{cases}
-1, & a<0, \\
0, & a=0, \\
1, & a>0.
\end{cases}
\label{eq:sgna}
\end{equation}

%...............................................................................
\begin{theorem}
Suppose $\det(A) \ne 0$, $\det(\tilde{M}) \ne 0$, $c_1 \ne 0$, and $\alpha_L \ne 0$.
Then
\begin{equation}
\label{eq:Feigin}								% <-- if placed below it catches the theorem counter!?
{\rm sgn} \left( \alpha_L \alpha_S \right) = (-1)^{N_L + N_S} \,{\rm sgn}(c_1),
\end{equation}
where $N_L$ is the number of real positive eigenvalues of $A$,
and $N_S$ is the number of real positive eigenvalues of $\tilde{M}$.
\label{th:Feigin}
\end{theorem}

Theorem \ref{th:Feigin} is practical in the sense that it
shows how the BEB can be classified from a simple calculation.
%In the remainder of this section we relate $N_L$ and $N_S$
%to the dimensions of the unstable manifolds of $x^L$ and $x^S$
%which, in practice, may be known from numerical simulations.
The matrices $A$ and $\tilde{M}$ govern the stability of $x^L$ and $x^S$ and their local dynamics.
This is because the Jacobian matrix $D F^L$ evaluated at $x^L(\mu)$ is equal to $A + \cO(\mu)$.
Thus, if $A$ has no eigenvalues with zero real part,
then, for all sufficiently small values of $\mu$ for which $x^L(\mu)$ is admissible,
the dimension of the unstable manifold of $x^L(\mu)$ is equal to the number of eigenvalues of $A$
(counting algebraic multiplicity) with positive real part.
If we let $D_L$ denote this dimension, then
\begin{equation}
(-1)^{D_L} = (-1)^{N_L},
\label{eq:DL}
\end{equation}
because complex eigenvalues of $A$ appear in complex conjugate pairs.

More care is required to describe the stability of the pseudo-equilibrium $x^S$ in the same manner.
In the context of the $(n-1)$-dimensional sliding dynamics,
if $\tilde{M}$ has no eigenvalues with zero real part,
then the dimension of the unstable manifold of $x^S(\mu)$, call it $\tilde{D}_S$,
is equal to the number of eigenvalues of $\tilde{M}$ (counting algebraic multiplicity) with positive real part,
and, as above, $(-1)^{\tilde{D}_S} = (-1)^{N_S}$.
For the full system \eqref{eq:F2},
we look at the type of sliding region to which $x^S$ belongs.
If $c_1 < 0$, then the sliding region is attracting,
and so the dimension of the unstable manifold of $x^S(\mu)$, call it $D_S$, is equal to $\tilde{D}_S$.
If instead $c_1 > 0$, then the sliding region is repelling, and so $D_S = \tilde{D}_S + 1$,
In summary,
\begin{equation}
(-1)^{D_S} = \begin{cases}
(-1)^{N_S}, & c_1 < 0, \\
(-1)^{N_S + 1}, & c_1 > 0.
\end{cases}
\label{eq:DS}
\end{equation}

By then combining \eqref{eq:Feigin}--\eqref{eq:DS} we obtain
\begin{equation}
{\rm sgn} \left( \alpha_L \alpha_S \right) = (-1)^{D_L + D_S + 1},
\label{eq:Feigin2}
\end{equation}
which connects the classification of the BEB to the dimensions of the unstable manifolds of the equilibria.
In practice these dimensions may be known from numerical simulations.
In particular, if both $x^L$ and $x^S$ are stable,
then $D_L = D_S = 0$, thus $\alpha_L \alpha_S < 0$,
and hence the BEB corresponds to persistence by Lemma \ref{le:classification}.

%===============================================================================
\section{Proof of Theorem \ref{th:Feigin}}
%\section{Calculations}
\label{sec:proof}
%\setcounter{equation}{0}

%Here we first connect the admissibility of $x^L(\mu)$ to $\det(A)$,
%we then connect the admissibility of $x^S(\mu)$ to that of $x^L(\mu)$,
%and we then connect the admissibility of $x^S(\mu)$ to $\det(\tilde{M})$.

The {\em adjugate} of $A$, denoted ${\rm adj}(A)$,
is the transpose of the cofactor matrix of $A$.
In particular, if $\det(A) \ne 0$, then ${\rm adj}(A) = \det(A) A^{-1}$.
As with continuous piecewise-smooth systems \cite{Si10}, we let
\begin{equation}
\varrho^{\sf T} = e_1^{\sf T} {\rm adj}(A).
\label{eq:varrho}
\end{equation}
Then, by \eqref{eq:xL} and \eqref{eq:alphaL},
%\begin{equation}
%x^L_1(\mu) = -\frac{\varrho^{\sf T} b}{\det(A)} \,\mu + \cO \left( \mu^2 \right).
%\label{eq:xL1}
%\end{equation}
\begin{equation}
\alpha_L = -\frac{\varrho^{\sf T} b}{\det(A)}.
\label{eq:alphaL2}
\end{equation}
Recall, we require $\det(A) \ne 0$ so that $x^L$ is well-defined and unique.
From \eqref{eq:alphaL2} we see that $\varrho^{\sf T} b \ne 0$ is needed to ensure
that $x^L(\mu)$ moves away from $\Sigma$ as the value of $\mu$ is varied from $0$
at an asymptotically linear rate.
That is, $\varrho^{\sf T} b \ne 0$ is the {\em transversality condition} for the BEB.

As mentioned in \sect{intro},
the distinction between persistence and a nonsmooth-fold
has previously been equated to the sign of a certain inner product \cite{DiBu08,DiPa08}.
For our system \eqref{eq:F2}, this inner product is $\varrho^{\sf T} c$
and appears in the following result.

%...............................................................................
\begin{lemma}
Suppose $\det(A) \ne 0$, $\det(\tilde{M}) \ne 0$, $c_1 \ne 0$, and $\varrho^{\sf T} c \ne 0$.
Then
\begin{equation}
F^L_1(x^S(\mu);\mu) = \frac{\varrho^{\sf T} b \,c_1}{\varrho^{\sf T} c} \,\mu + \cO \left( \mu^2 \right).
\label{eq:FL1xS}
\end{equation}
\label{le:FL1xS}
\end{lemma}

The following proof of Lemma \ref{le:FL1xS}
closely follows that of di Bernardo {\em et.~al.}~\cite{DiBu08}.

%...............................................................................
\begin{proof}
The pseudo-equilibrium satisfies $F^S(x^S(\mu);\mu) = \b0$.
By \eqref{eq:Fslide3} we can rewrite $F^S$ as
\begin{equation}
F^S(x;\mu) = F^L(x;\mu) - \frac{F^L_1(x;\mu)}{c_1} \,c + \cO(2).
%\label{eq:Fslide4}
\end{equation}
By multiplying this on the left by $\varrho^{\sf T}$ and substituting $x = x^S(\mu)$ we obtain
\begin{equation}
0 = \varrho^{\sf T} F^L(x^S(\mu);\mu) - \frac{F^L_1(x^S(\mu);\mu) \varrho^{\sf T} c}{c_1} + \cO \left( \mu^2 \right).
\label{eq:FL1xSProof1}
\end{equation}
The first term in \eqref{eq:FL1xSProof1} simplifies to $\varrho^{\sf T} b \mu$ (because $x^S_1(\mu) = 0$)
and so the desired expression \eqref{eq:FL1xS} results from a simple rearrangement of \eqref{eq:FL1xSProof1}.
\end{proof}

The next result, which to the best of the author's knowledge is new,
shows how the inner product $\varrho^{\sf T} c$ is related to the Jacobian matrix $\tilde{M}$.

%...............................................................................
\begin{lemma}
Suppose $\det(A) \ne 0$ and $c_1 \ne 0$.
Then
\begin{equation}
\det(\tilde{M}) = \frac{\varrho^{\sf T} c}{c_1}.
\label{eq:dettildeM}
\end{equation}
\label{le:dettildeM}
\end{lemma}

%...............................................................................
\begin{proof}
Since the first row of $M$ is $\b0^{\sf T}$,
the characteristic polynomials of $M$ and $\tilde{M}$ are related by
\begin{equation}
\det(\lambda I - M) = \lambda \det(\lambda I - \tilde{M}).
\label{eq:dettildeMProof1}
\end{equation}
To evaluate $\det(\lambda I - M)$, we use the matrix determinant lemma:
$\det(X + v u^{\sf T}) = \det(X) \left( 1 + u^{\sf T} X^{-1} v \right)$,
for any non-singular $n \times n$ matrix $X$ and $u, v \in \mathbb{R}^n$.
From \eqref{eq:M} we obtain
\begin{equation}
\det(\lambda I - M) = \det(\lambda I - A)
\left( 1 + \frac{e_1^{\sf T} A (\lambda I - A)^{-1} c}{c_1} \right).
%\label{eq:dettildeMProof2}
\nonumber
\end{equation}
By then substituting $(\lambda I - A)^{-1} = -A^{-1} - A^{-2} \lambda + \cO \left( \lambda^2 \right)$,
we produce
\begin{equation}
\det(\lambda I - M) = -\frac{\det(-A)}{c_1}
\,e_1^{\sf T} A^{-1} c \lambda + \cO \left( \lambda^2 \right).
%\label{eq:dettildeMProof3}
\nonumber
\end{equation}
By \eqref{eq:varrho} this reduces to
\begin{equation}
\det(\lambda I - M) = \frac{(-1)^{n+1} \varrho^{\sf T} c}{c_1} \,\lambda + \cO \left( \lambda^2 \right).
%\label{eq:dettildeMProof4}
\nonumber
\end{equation}
Thus by \eqref{eq:dettildeMProof1} we have
$\det(-\tilde{M}) = \frac{(-1)^{n+1} \varrho^{\sf T} c}{c_1}$,
and hence \eqref{eq:dettildeM}, as required.
\end{proof}

By combining \eqref{eq:FL1xS}, \eqref{eq:dettildeM}, and
$F^R_1(x^S(\mu);\mu) = c_1 + \cO(\mu)$, we arrive at
\begin{equation}
\alpha_S = \frac{\varrho^{\sf T} b \,c_1}{\det(\tilde{M})}.
\label{eq:alphaR2}
\end{equation}
To complete the proof of Theorem \ref{th:Feigin}, we combine \eqref{eq:alphaL2} and \eqref{eq:alphaR2} to obtain
%\begin{equation}
%\alpha_L \alpha_S = \frac{-\left( \varrho^{\sf T} b \right)^2 c_1}{\det(A) \det(\tilde{M})}.
%\nonumber
%\end{equation}
\begin{equation}
{\rm sgn} \left( \alpha_L \alpha_S \right) =
-{\rm sgn} \left( \det(A) \det(\tilde{M}) \,c_1 \right).
\label{eq:FeiginProof9}
\end{equation}
Notice $(-1)^{N_L} = {\rm sgn}(\det(-A))$,
and $(-1)^{N_S} = {\rm sgn}(\det(-\tilde{M}))$.
Also $\det(A) = (-1)^n \det(-A)$, and $\det(\tilde{M}) = (-1)^{n+1} \det(-\tilde{M})$.
Therefore \eqref{eq:FeiginProof9} is equivalent to the given formula \eqref{eq:Feigin}.

%===============================================================================
\section{A remark on codimension-two BEBs}
\label{sec:codim2}

The above results add insight into some codimension-two BEBs.
Suppose the two-dimensional parameter space of a Filippov system has a curve of BEBs.
Further suppose that at a point on this curve the BEBs change from persistence to a nonsmooth-fold.
That is, one of $\alpha_L$ and $\alpha_S$ changes sign.
The case that $\alpha_L$ changes sign was unfolded for a simple system
by di Bernardo {\em et.~al.}~\cite{DiNo08};
the case that $\alpha_S$ changes sign was unfolded in a general setting
by Della Rossa and Dercole \cite{DeDe12}. % repeating [DeDe11b]
In both cases a curve of saddle-node bifurcations emanates from the codimension-two point with a quadratic tangency.
This is illustrated in Fig.~\ref{fig:unfoldingBEBSN_framework}.
The same unfolding occurs for the analogous scenario in continuous piecewise-smooth systems \cite{SiKo09}.

%%%%%%%%%%%%%%%%%%%%%%%%%%%%%%%%%%%%%%%%%%%%%%%%%%%%%%%%%%%%%%%%%%%%%%%%%%%%%%%%%%%%%%%%%%%%%%%%%%%%%%%%%%%%%%%%%%%%%%%%
\begin{figure}[h!]
\begin{center}
\setlength{\unitlength}{1cm}
\begin{picture}(5.2,3.9)
\put(0,0){\includegraphics[width=5.2cm]{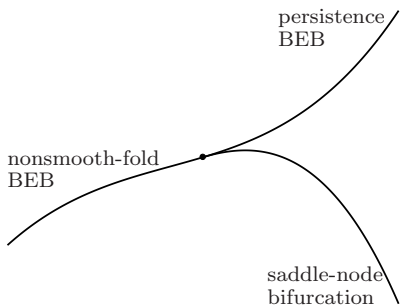}}
\put(0,1.85){\scriptsize nonsmooth-fold}
\put(0,1.55){\scriptsize BEB}
\put(3.6,3.73){\scriptsize persistence}
\put(3.6,3.43){\scriptsize BEB}
\put(3.46,.3){\scriptsize saddle-node}
\put(3.46,0){\scriptsize bifurcation}
\end{picture}
\caption{
A sketch of a typical two-parameter bifurcation diagram of a Filippov system
about a codimension-two point where a curve of BEBs changes from persistence to nonsmooth-fold type.
\label{fig:unfoldingBEBSN_framework}
} 
\end{center}
\end{figure}
%%%%%%%%%%%%%%%%%%%%%%%%%%%%%%%%%%%%%%%%%%%%%%%%%%%%%%%%%%%%%%%%%%%%%%%%%%%%%%%%%%%%%%%%%%%%%%%%%%%%%%%%%%%%%%%%%%%%%%%%

These unfoldings assume that in a neighbourhood of the codimension-two point
the transversality condition of the BEBs is satisfied (i.e.~$\varrho^{\sf T} b \ne 0$).
It is also assumed that, locally, the system without a regular equilibrium
has no tangency with the discontinuity surface (i.e.~$c_1 \ne 0$).
Then by \eqref{eq:alphaL2} and \eqref{eq:alphaR2},
$\alpha_L$ and $\alpha_S$ must change sign, not by becoming zero,
but rather by going `through infinity' by either $\det(A) = 0$ or $\det(\tilde{M}) = 0$.
This tells us that at the codimension-two point one of the equilibria has a zero eigenvalue,
and this provides an explanation for the presence of saddle-node bifurcations.

%===============================================================================
\section{A normal form}
\label{sec:normalForm}

Here we introduce the normal form
\begin{equation}
\dot{x} = \begin{cases}
C x + e_n \mu, & x_1 < 0, \\
d, & x_1 > 0,
\end{cases}
\label{eq:normalForm}
\end{equation}
where $C$ is the companion matrix
\begin{equation}
C = \begin{bmatrix}
-a_1 & 1 \\
-a_2 && \ddots \\
\vdots &&& 1 \\
-a_n
\end{bmatrix},
\label{eq:companionMatrix}
\end{equation}
and $d \in \mathbb{R}^n$ with $d_1 = \pm 1$.
The $a_i \in \mathbb{R}$ are the coefficients of the characteristic polynomial of $C$:
\begin{equation}
\det(\lambda I - C) = \lambda^n + a_1 \lambda^{n-1} + \cdots + a_{n-1} \lambda + a_n \,.
\label{eq:ai}
\end{equation}

The following theorem justifies our interpretation of \eqref{eq:normalForm} as a `normal form' for BEBs
by providing a coordinate transformation from a system in the general form \eqref{eq:F2} to \eqref{eq:normalForm}.
Since \eqref{eq:F2} is piecewise-smooth, it may tempting to apply different coordinate transformations
to the two pieces of \eqref{eq:F2}.
However, this may alter the sliding dynamics in a fundamental way.
%\removableFootnote{
%In [Gl16d], preservation of the sliding dynamics is discussed,
%but a piecewise-coordinate change is not.
%},
Indeed, as we will see, a single coordinate transformation is sufficient.
%Here we apply a single coordinate transformation
%to ensure that the original system and the transformed system are topologically equivalent.
In the following result, $J$ denotes the companion matrix \eqref{eq:companionMatrix}
for which $a_i = 0$ for all $i$.
%$\det(\lambda I - J) = \lambda^n$.

%...............................................................................
\begin{theorem}
Consider a system of the form \eqref{eq:F2} with $c_1 \ne 0$.
Let $a_1,\ldots,a_n \in \mathbb{R}$ be the coefficients of the characteristic polynomial of $A$
(matching \eqref{eq:ai}).
Let
\begin{equation}
\begin{split}
\Psi &= \begin{bmatrix}
1 \\
a_1 & 1 \\
\vdots & \ddots & \ddots \\
a_{n-1} & \cdots & a_1 & 1
\end{bmatrix}, \\
\Phi &= \begin{bmatrix}
e_1^{\sf T} \\
e_1^{\sf T} A \\
\vdots \\
e_1^{\sf T} A^{n-1}
\end{bmatrix},
\end{split}
\label{eq:PsiPhi}
\end{equation}
and
\begin{equation}
\begin{split}
Q &= \Psi \Phi, \\
r &= J^{\sf T} Q b, \\
s &= e_n^{\sf T} Q b.
\end{split}
\label{eq:Qrs}
\end{equation}
If $\det(\Phi) \ne 0$ and $s \ne 0$, then under the coordinate transformation
\begin{equation}
\begin{split}
x &\mapsto Q x + r \mu, \\
\mu &\mapsto s \mu,
\end{split}
\label{eq:transocf}
\end{equation}
the system \eqref{eq:F2} becomes
%\removableFootnote{
%For my dimension reduction paper I implemented this in {\sc transToOCF2.m}.
%}
\begin{equation}
\dot{x} = \begin{cases}
C x + e_n \mu + \cO(2), & x_1 < 0, \\
d + \cO(1), & x_1 > 0,
\end{cases}
\label{eq:transformedSystem}
\end{equation}
where $C$ is given by \eqref{eq:companionMatrix} and $d = Q c$.
The additional transformation, $x \mapsto \frac{1}{|d_1|} x$ and
$\mu \mapsto \frac{1}{|d_1|} \mu$,
scales the first element of $d$ to $\pm 1$.
\label{th:normalForm}
\end{theorem}

The normal form \eqref{eq:normalForm} results from removing
higher order terms from \eqref{eq:transformedSystem}.
We do not provide a proof of Theorem \ref{th:normalForm}
as it is a trivial generalisation of that for continuous piecewise-smooth systems \cite{Si18}.
%(see \cite{Si18} for a proof in the continuous setting).
The use of companion matrices stems from control theory \cite{DiBu08}.
Indeed, $\Phi$ is called an {\em observability matrix}
and, as in the continuous setting,
we say that the system \eqref{eq:F2} is `observable' if $\det(\Phi) \ne 0$.

It is typical for a given system \eqref{eq:F2} to be observable.
Indeed the `Popov-Belevitch-Hautus observability test' \cite{So98}
(a control theory result)
tells us that \eqref{eq:F2} is observable if and only if
$A$ does not have an eigenvector orthogonal to $e_1$ \cite{Si16}.
For the unfolding of a codimension-two BEB at which
$A$ has an eigenvector orthogonal to $e_1$,
refer to Section 8 of Guardia {\em et.~al.}~\cite{GuSe11}.

Note that $e_1^{\sf T} Q = e_1^{\sf T}$ and $r_1 = 0$,
due to the way $\Psi$ and $\Phi$ are defined.
Hence the coordinate transformation \eqref{eq:transocf} leaves $x_1$ unchanged, and $d_1 = c_1$.
Thus the assumption $c_1 \ne 0$ ensures that the additional transformation
in Theorem \ref{th:normalForm} is well-defined.

%An alternate and more explicit formula for $Q$ is given in \cite{Si10,Si16}.

%===============================================================================
\section{Properties of the normal form}
\label{sec:properties}

Here we study \eqref{eq:normalForm} (where $d_1 = \pm 1$).
If $a_n \ne 0$ (equivalently, if $\det(C) \ne 0$),
then \eqref{eq:normalForm} has the unique regular equilibrium
$x^L(\mu) = -e_1^{\sf T} C^{-1} e_n \mu$.
In particular, $x^L_1(\mu) = \frac{1}{a_n} \,\mu$,
%(this simplification is a consequence of the form \eqref{eq:companionMatrix}),
thus $x^L$ moves linearly away from the discontinuity surface $\Sigma$ as $\mu$ is varied from $0$.
That is, the transversality condition is automatically satisfied in the normal form.
This tells us that if the coordinate transformation from \eqref{eq:F2} to \eqref{eq:normalForm} can be achieved,
then the transversality condition needs to be satisfied
in \eqref{eq:F2} (this condition is $\varrho^{\sf T} b \ne 0$).
Indeed, via direct calculations and the Cayley Hamilton theorem,
it can be shown that $s = (-1)^{n+1} \varrho^{\sf T} b$,
and $s \ne 0$ is required in Theorem \ref{th:normalForm}.

Next we study the sliding dynamics of \eqref{eq:normalForm}.
For the $x_1 < 0$ component of \eqref{eq:normalForm},
with $x_1 = 0$ we have $\dot{x}_1 = x_2$.
Thus if $d_1 = -1$, then $x_2 > 0$ is an attracting sliding region and $x_2 < 0$ is a crossing region.
If instead $d_1 = 1$, then $x_2 > 0$ is a crossing region and $x_2 < 0$ is a repelling sliding region.

Sliding dynamics are governed by \eqref{eq:Fslide2} applied to \eqref{eq:normalForm}.
If we are only interested in the paths that orbits take, not evolution times,
we can scale time in a spatially dependent way so that the denominator of \eqref{eq:Fslide2}
changes from $F^L_1 - F^R_1$ to $-F^R_1$.
This is a common strategy for dealing with sliding motion \cite{Fi88},
and particularly beneficial here as the resulting scaled sliding vector field is linear, specifically:
\begin{equation}
\begin{bmatrix} \dot{x}_2 \\ \vdots \\ \dot{x}_n \end{bmatrix} =
\tilde{M} \begin{bmatrix} x_2 \\ \vdots \\ x_n \end{bmatrix} +
\begin{bmatrix} 0 \\ \vdots \\ 0 \\ \mu \end{bmatrix},
\label{eq:FslideNormalForm2}
\end{equation}
where
\begin{equation}
\tilde{M} = \begin{bmatrix}
-\frac{d_2}{d_1} & 1 \\
-\frac{d_3}{d_1} && \ddots \\
\vdots &&& 1 \\
-\frac{d_n}{d_1}
\end{bmatrix}.
\label{eq:tildeM}
\end{equation}
Notice that \eqref{eq:FslideNormalForm2} has the same form as
the $x_1 < 0$ component of \eqref{eq:normalForm}, except it is of one less dimension.
In particular, the coefficients of the characteristic polynomial of $\tilde{M}$ are
$\frac{d_2}{d_1},\ldots,\frac{d_n}{d_1}$.

%\begin{equation}
%\begin{bmatrix} x_2' \\ \vdots \\ x_n' \end{bmatrix} =
%\tilde{M} \begin{bmatrix} x_2 \\ \vdots \\ x_n \end{bmatrix} +
%e_{n-1} \mu
%\label{eq:FslideNormalForm2}
%\end{equation}

Sliding motion ceases at $x_2 = 0$.
Here \eqref{eq:FslideNormalForm2} has $\dot{x}_2 = x_3$ (assuming $n \ge 3$).
Thus the direction of sliding motion at $x_2 = 0$ is governed by the sign of $x_3$.
In particular, if $d_1 = -1$ then sliding orbits can only escape the attracting sliding region $x_2 > 0$
at points on $x_2 = 0$ with $x_3 < 0$.

Since \eqref{eq:normalForm} and \eqref{eq:FslideNormalForm2} have no quadratic or higher order terms,
the structure of the dynamics is independent of the magnitude of $\mu$.
All bounded invariant sets collapse linearly to the origin as $\mu \to 0$,
and to understand the dynamics it suffices to consider $\mu \in \{ -1, 0, 1 \}$.

%===============================================================================
\section{BEBs in two dimensions}
\label{sec:2d}
%\setcounter{equation}{0}

%%%%%%%%%%%%%%%%%%%%%%%%%%%%%%%%%%%%%%%%%%%%%%%%%%%%%%%%%%%%%%%%%%%%%%%%%%%%%%%%%%%%%%%%%%%%%%%%%%%%%%%%%%%%%%%%%%%%%%%%
\begin{figure*}
\begin{center}
\setlength{\unitlength}{1cm}
\begin{picture}(13.7,18.4)
\put(3.3,17.8){\small \parbox{40mm}{
$d_2 < 0$\\ (sliding motion approaches the origin)}}
\put(9,17.8){\small \parbox{40mm}{
$d_2 > 0$\\ (sliding motion heads away from the origin)}}
\put(2.8,13.2){\includegraphics[width=5.2cm]{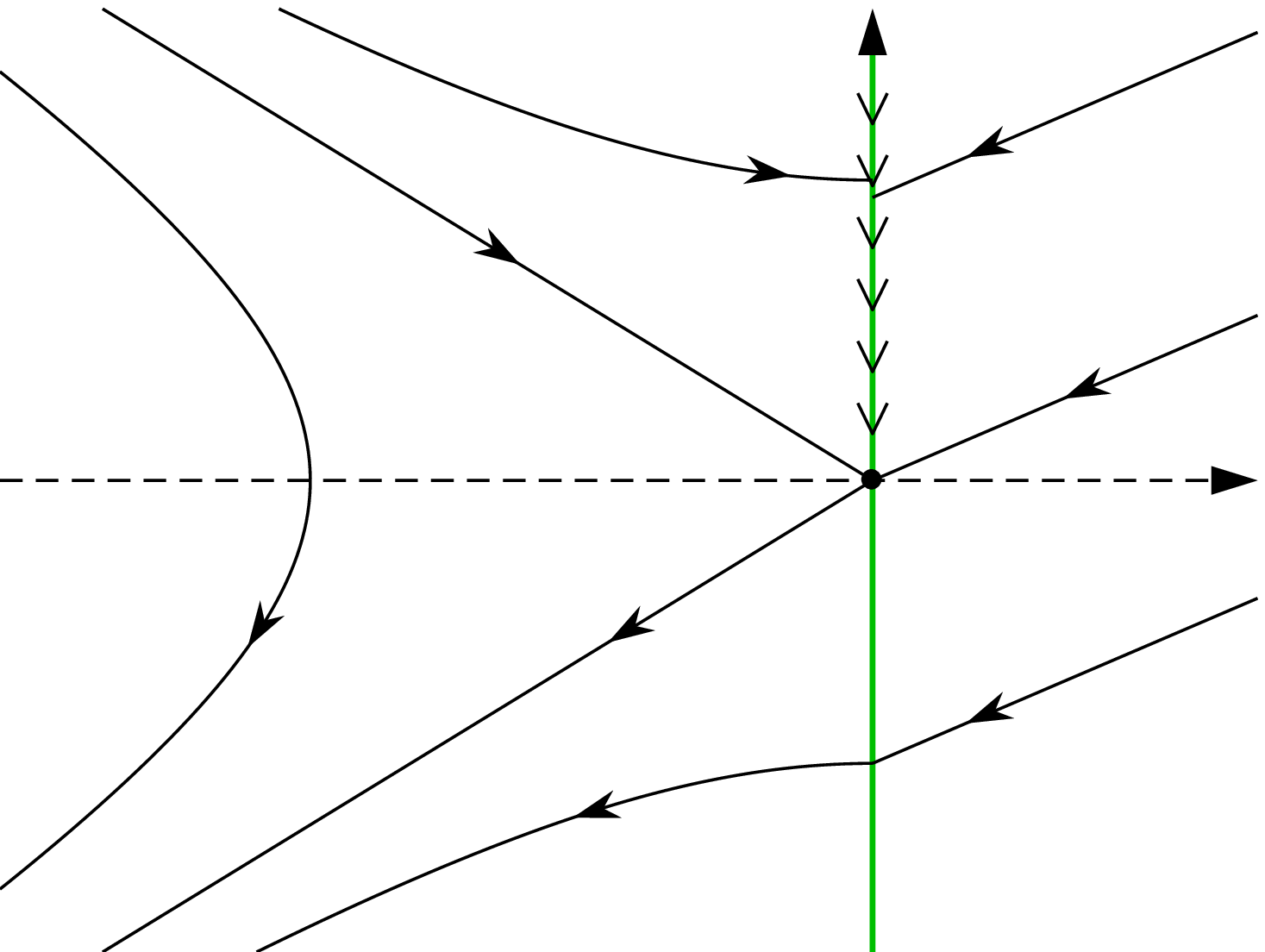}}
\put(8.5,13.2){\includegraphics[width=5.2cm]{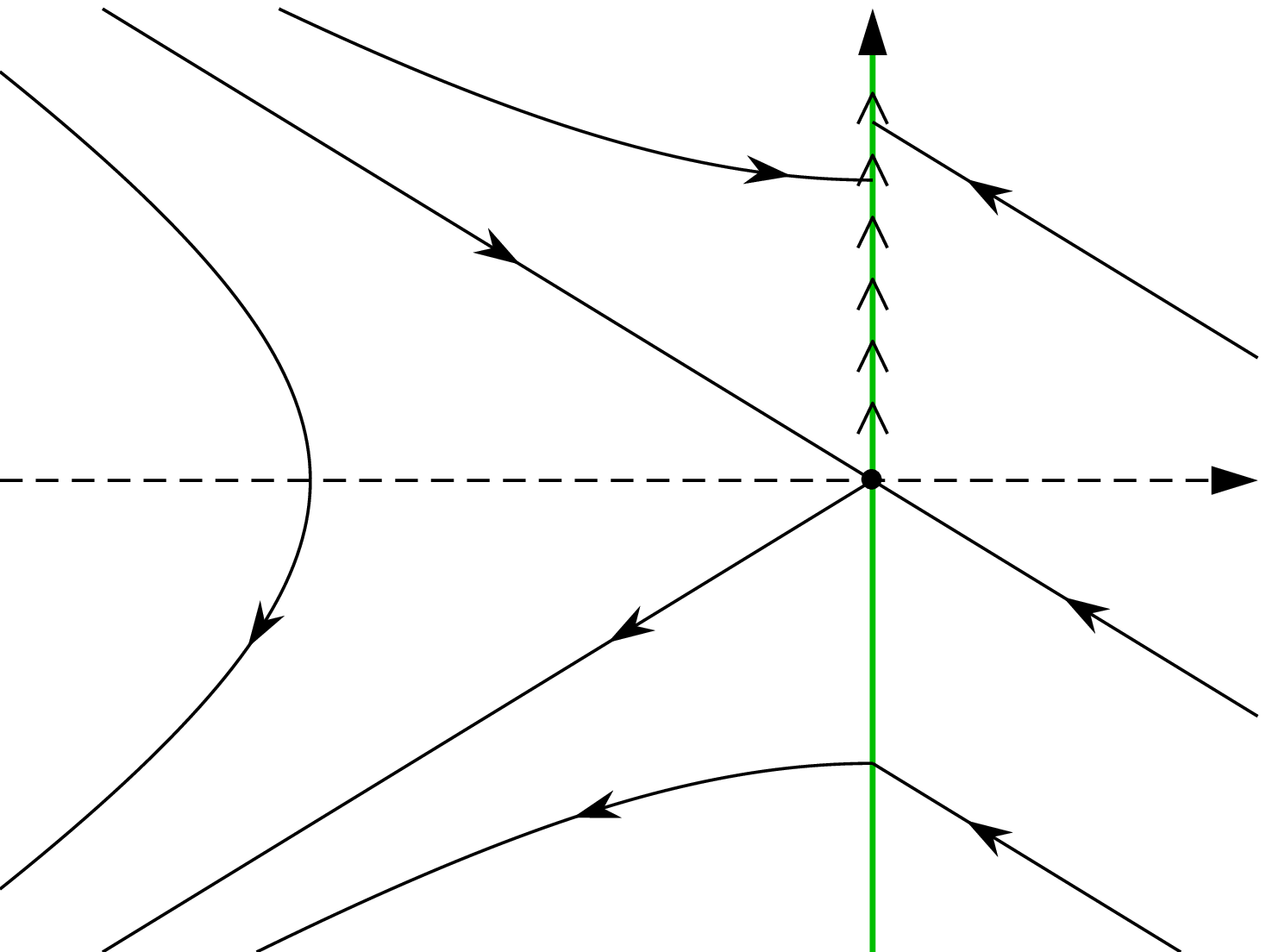}}
\put(0,15.32){\small $\delta_L < 0$}
\put(0,14.92){\small (saddle)}
\put(7.7,15.28){\footnotesize $x_1$}
\put(6.5,16.85){\footnotesize $x_2$}
\put(13.4,15.28){\footnotesize $x_1$}
\put(12.2,16.85){\footnotesize $x_2$}
\put(2.8,8.8){\includegraphics[width=5.2cm]{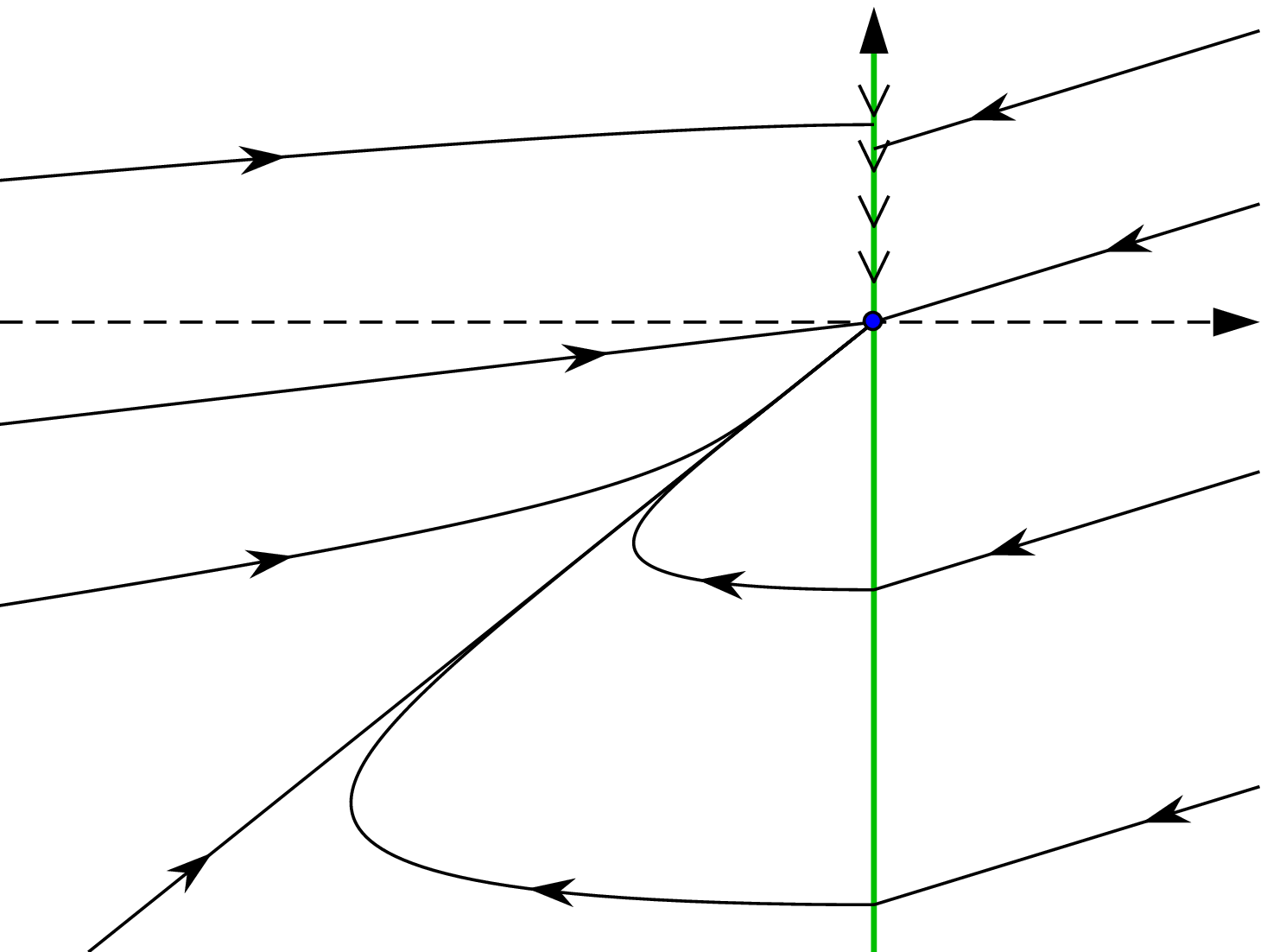}}
\put(8.5,8.8){\includegraphics[width=5.2cm]{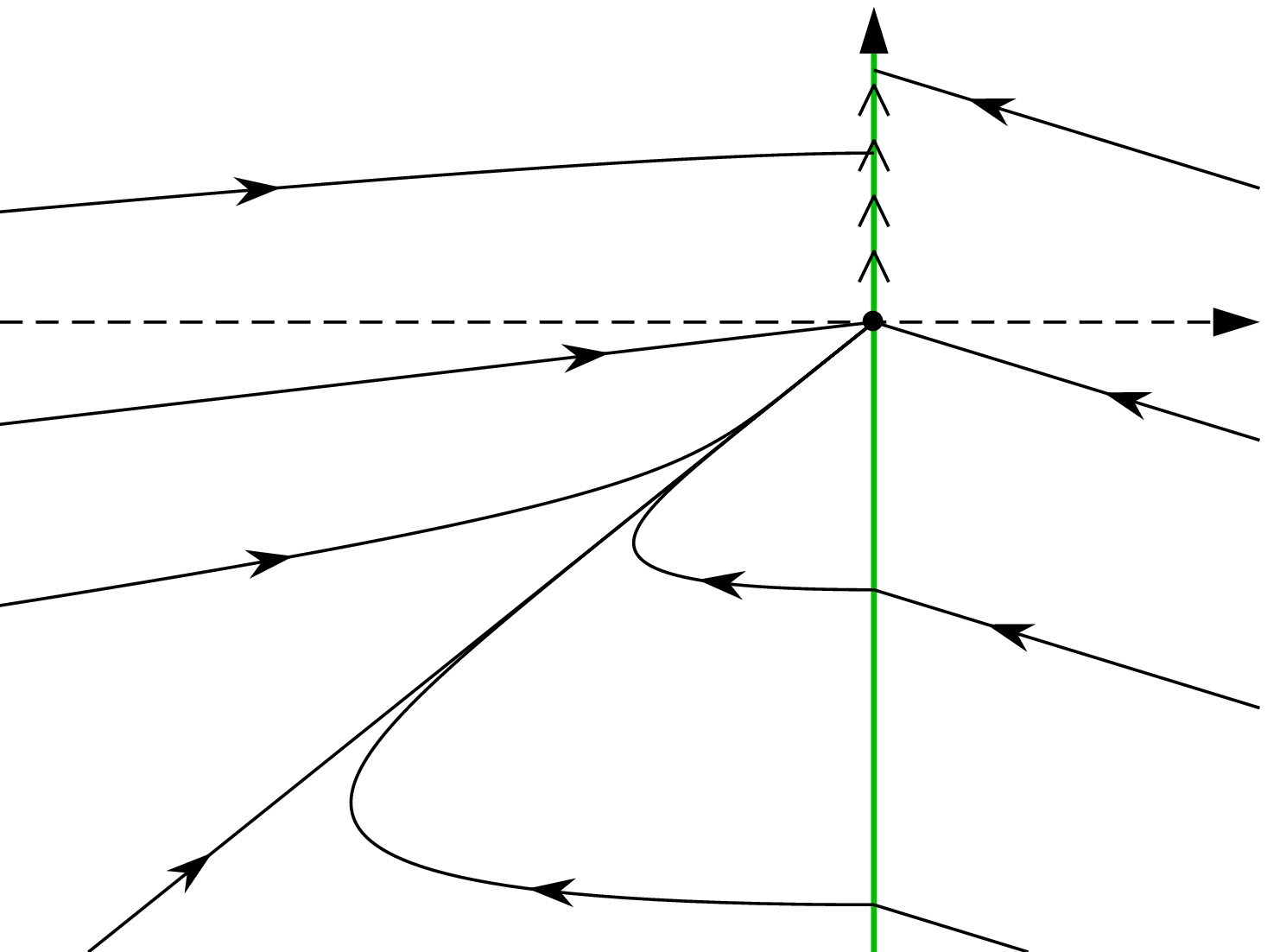}}
\put(0,11.12){\small $\tau_L < 0$}
\put(0,10.72){\small $0 < \delta_L < \frac{\tau_L^2}{4}$}
\put(0,10.32){\small (attracting node)}
\put(7.7,11.53){\footnotesize $x_1$}
\put(6.5,12.45){\footnotesize $x_2$}
\put(13.4,11.53){\footnotesize $x_1$}
\put(12.2,12.45){\footnotesize $x_2$}
\put(2.8,4.4){\includegraphics[width=5.2cm]{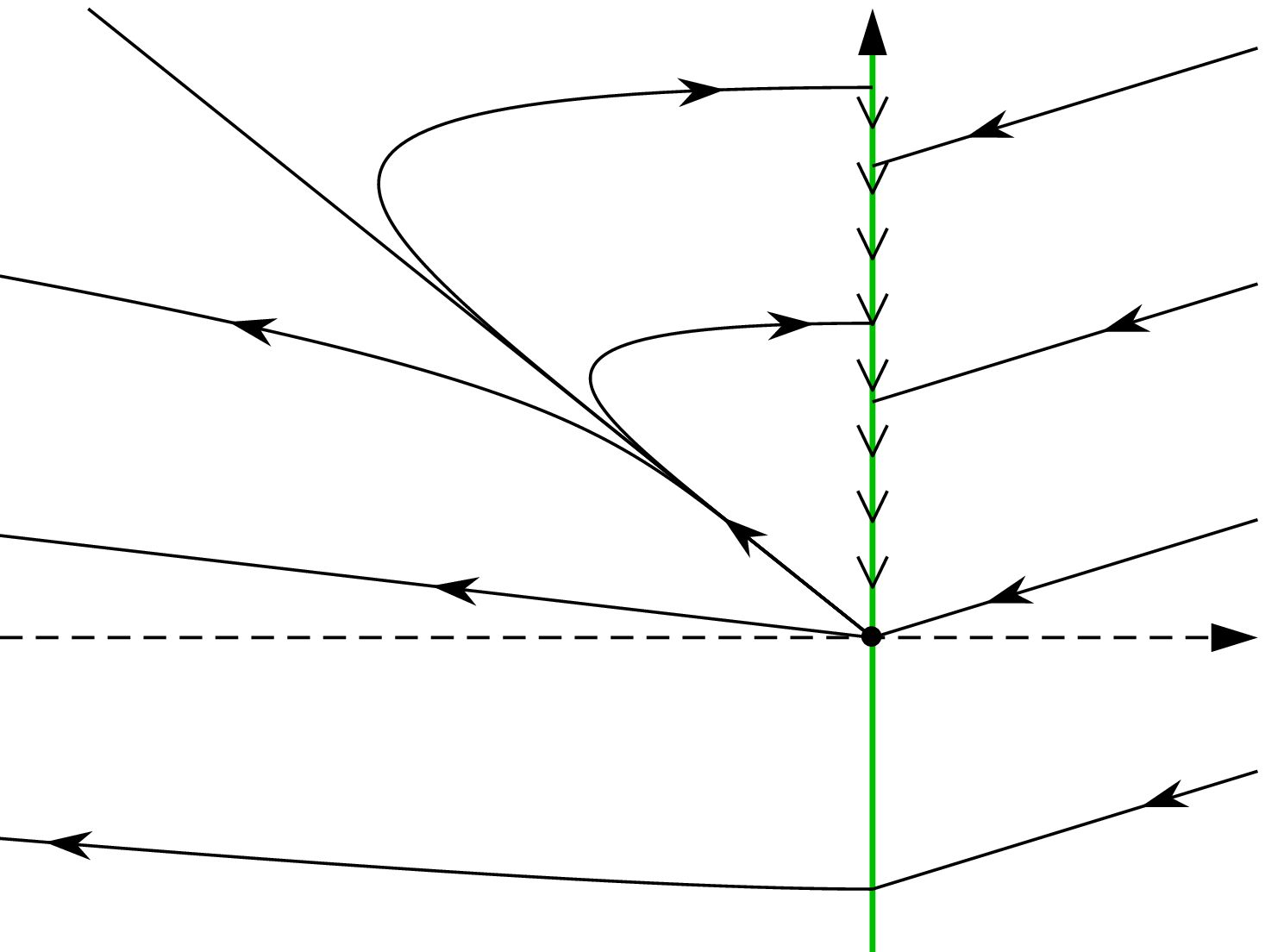}}
\put(8.5,4.4){\includegraphics[width=5.2cm]{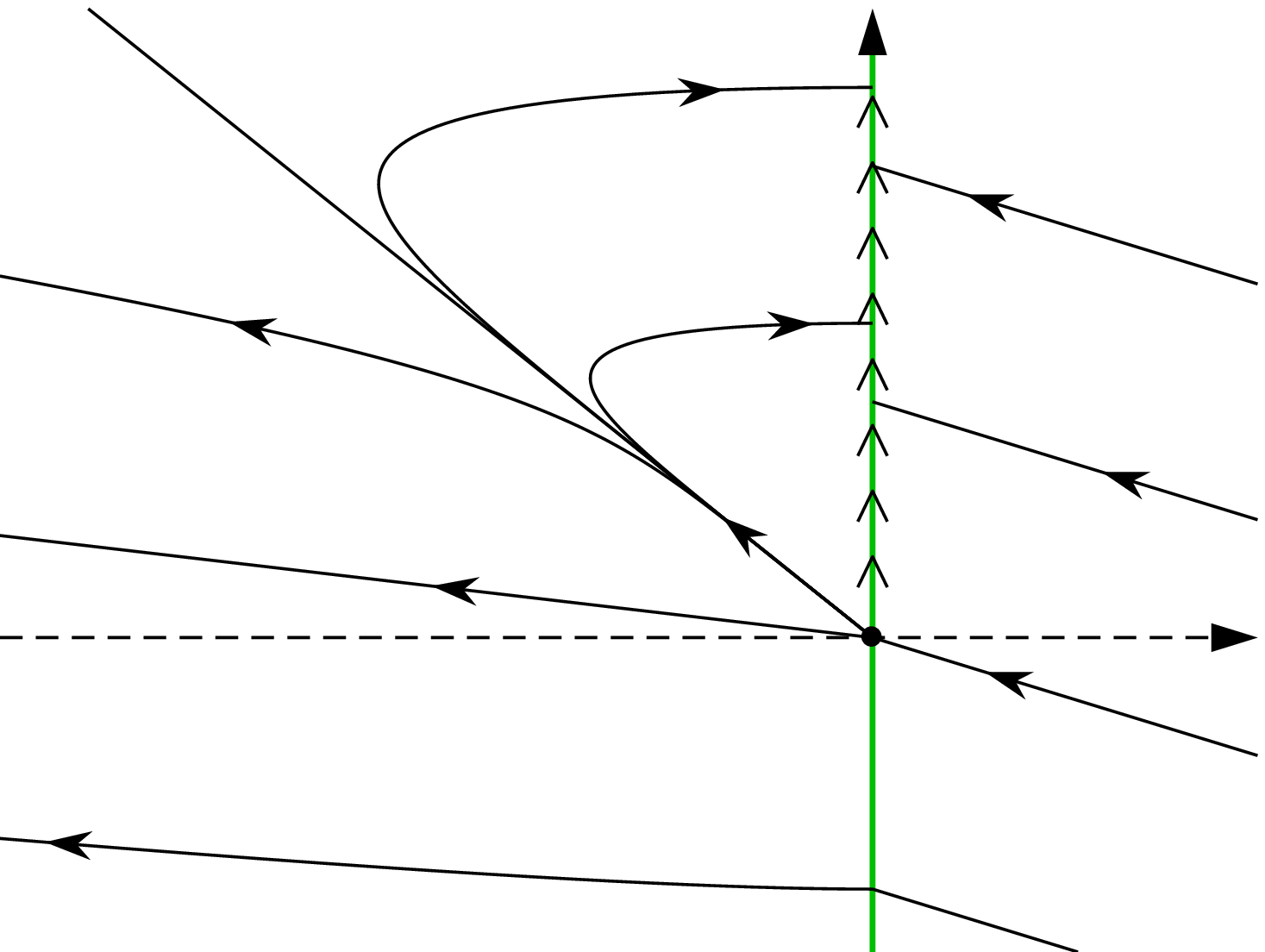}}
\put(0,6.72){\small $\tau_L > 0$}
\put(0,6.32){\small $0 < \delta_L < \frac{\tau_L^2}{4}$}
\put(0,5.92){\small (repelling node)}
\put(7.7,5.83){\footnotesize $x_1$}
\put(6.5,8.05){\footnotesize $x_2$}
\put(13.4,5.83){\footnotesize $x_1$}
\put(12.2,8.05){\footnotesize $x_2$}
\put(2.8,0){\includegraphics[width=5.2cm]{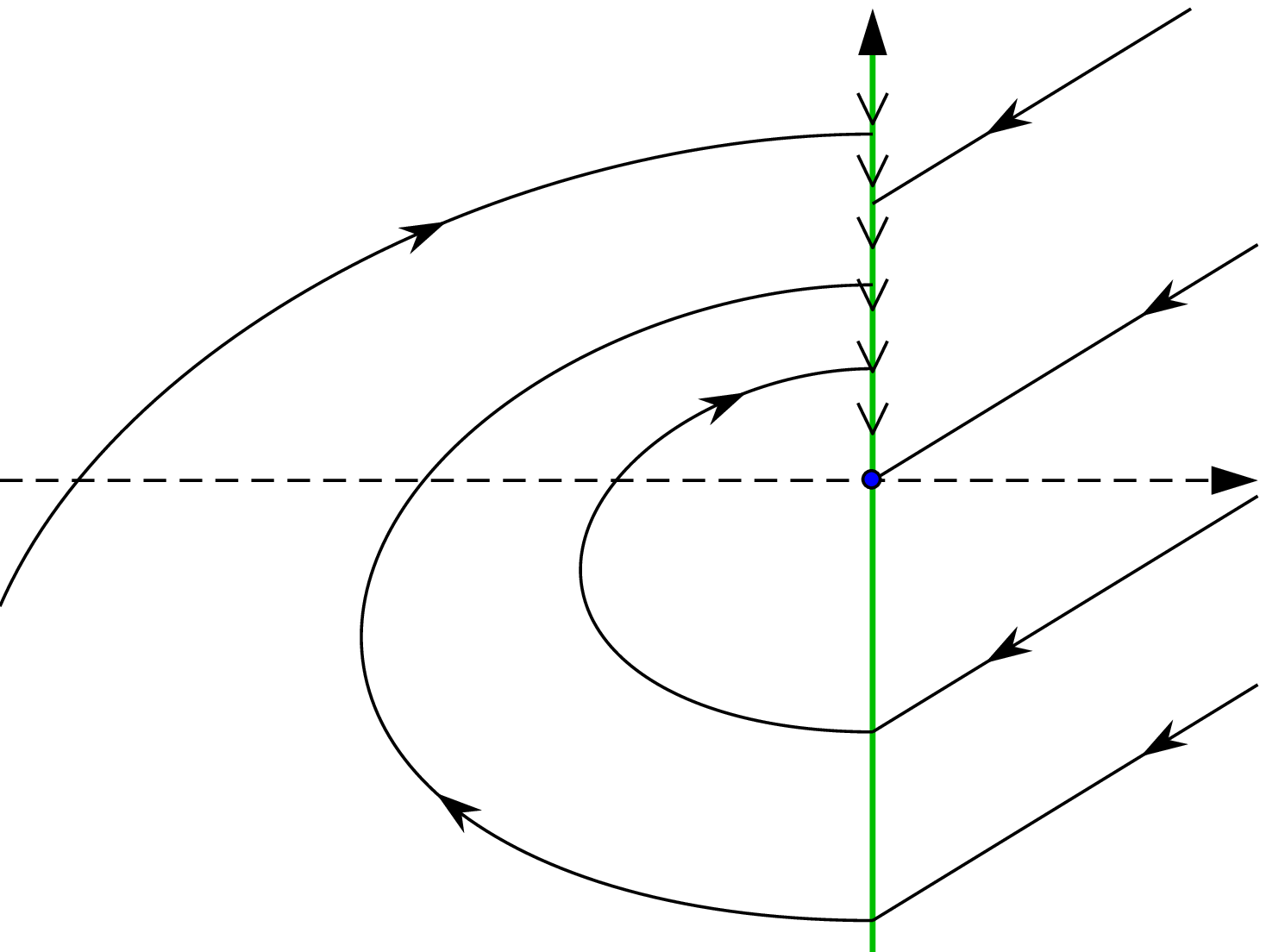}}
\put(8.5,0){\includegraphics[width=5.2cm]{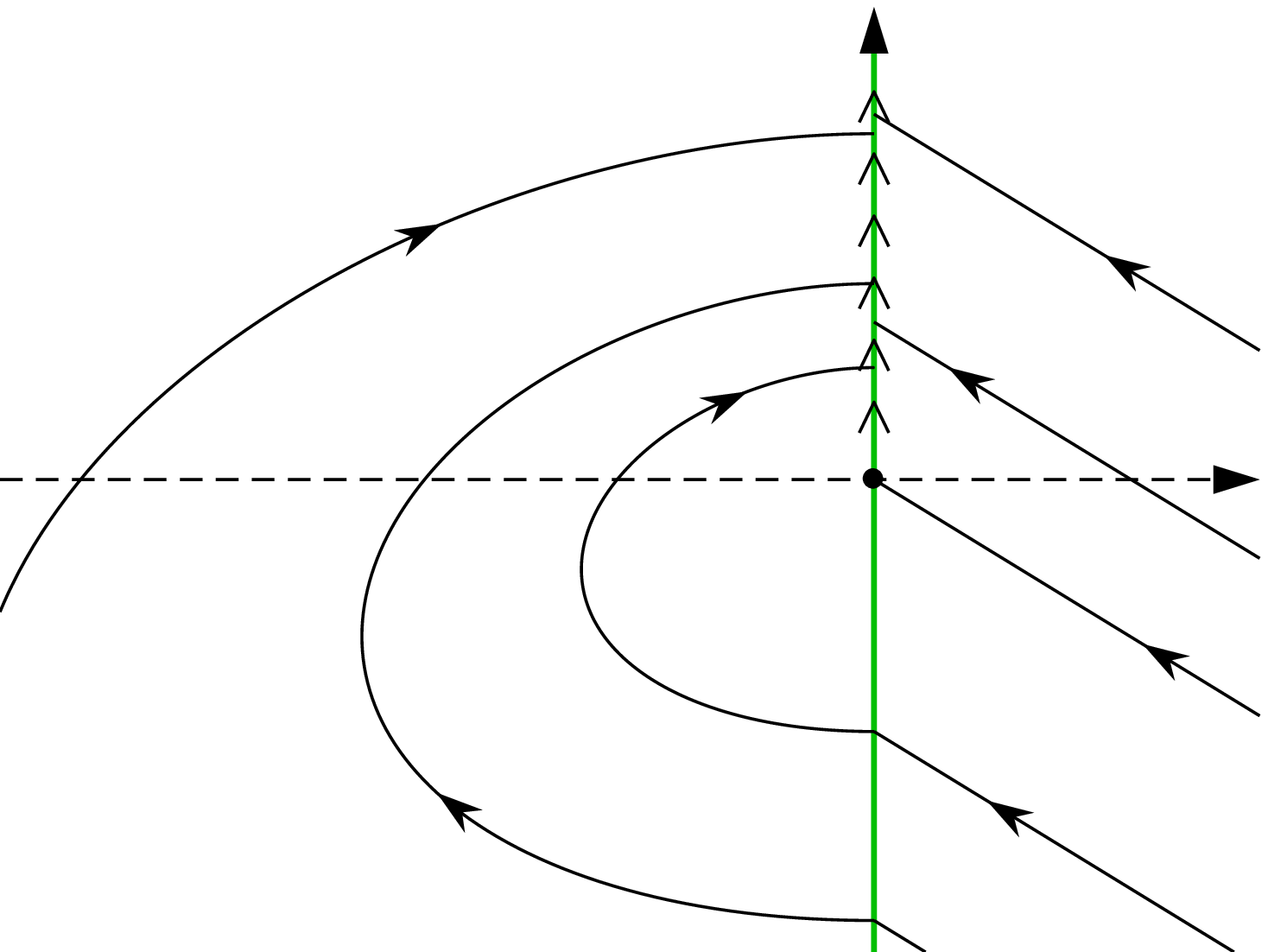}}
\put(0,2.12){\small $\delta_L > \frac{\tau_L^2}{4}$}
\put(0,1.72){\small (focus)}
\put(7.7,2.08){\footnotesize $x_1$}
\put(6.5,3.65){\footnotesize $x_2$}
\put(13.4,2.08){\footnotesize $x_1$}
\put(12.2,3.65){\footnotesize $x_2$}
\end{picture}
\caption{Filippov's eight scenarios for the dynamics local to a non-degenerate boundary equilibrium in two dimensions
as exemplified by the normal form \eqref{eq:normalForm} with $d_1 = -1$.
\label{fig:ppBEBFilippov_framework}
} 
\end{center}
\end{figure*}
%%%%%%%%%%%%%%%%%%%%%%%%%%%%%%%%%%%%%%%%%%%%%%%%%%%%%%%%%%%%%%%%%%%%%%%%%%%%%%%%%%%%%%%%%%%%%%%%%%%%%%%%%%%%%%%%%%%%%%%%

Here we consider \eqref{eq:normalForm} in two dimensions with $d_1 = -1$
(the case $d_1 = 1$ can be understood via a reversal of time).
We write
\begin{equation}
C = \begin{bmatrix} \tau_L & 1 \\ -\delta_L & 0 \end{bmatrix},
\nonumber
\end{equation}
so that $\tau_L$ and $\delta_L$ are the trace and determinant of $C$.
By \eqref{eq:tildeM}, we have $\tilde{M} = d_2$ (a scalar).
In summary, \eqref{eq:normalForm} has three parameters ($\tau_L, \delta_L, d_2 \in \mathbb{R}$)
in addition to the BEB parameter $\mu \in \mathbb{R}$.
The scaled sliding vector field is
\begin{equation}
\dot{x}_2 = d_2 x_2 + \mu.
\nonumber
\end{equation}

With $\mu = 0$, the origin is a boundary equilibrium.
In two dimensions there are eight topologically distinct non-degenerate scenarios
for the dynamics local to a boundary equilibrium \cite{Fi88}.
These are nicely characterised by the parameters in our normal form.
%Specifically, the values of $\tau_L$ and $\delta_L$
%give four generic cases for the nature of
The regular equilibrium $x^L$ is
(i) a saddle if $\delta_L < 0$,
(ii) an attracting node if $\tau_L < 0$ and $0 < \delta_L < \frac{\tau_L^2}{4}$,
(iii) a repelling node if $\tau_L > 0$ and $0 < \delta_L < \frac{\tau_L^2}{4}$, and
(iv) a focus if $\delta_L > \frac{\tau_L^2}{4}$.
Sliding motion is directed towards the origin if $d_2 < 0$,
and away from the origin if $d_2 > 0$.
These two sets of criteria are independent and
by combining them we obtain the eight generic scenarios illustrated in Fig.~\ref{fig:ppBEBFilippov_framework}.

These eight boundary equilibria can be unfolded by varying the value of $\mu$ in the normal form.
Some of these have multiple unfoldings.
For instance the lower-left boundary equilibrium of Fig.~\ref{fig:ppBEBFilippov_framework} has two generic unfoldings:
an attracting limit cycle exists for $\mu < 0$ if the focus is repelling;
no limit cycle exists if the focus is attracting.
The upper-left boundary equilibrium has two unfoldings,
the lower-right boundary equilibrium has three unfoldings,
and the remainder have one unfolding (see \cite{Gl16d,HoHo16} for details).
Thus there are $12$ generic BEBs in two dimensions.

%===============================================================================
\section{BEBs in three dimensions}
\label{sec:3d}

Here we consider \eqref{eq:normalForm} in three dimensions with $d_1 = -1$.
We write
\begin{equation}
C = \begin{bmatrix} \tau_L & 1 & 0 \\ -\sigma_L & 0 & 1 \\ \delta_L & 0 & 0 \end{bmatrix},
\nonumber
\end{equation}
so that $\tau_L$, $\sigma_L$, and $\delta_L$ are the
trace, second trace, and determinant of $C$ (see Appendix C of Simpson \cite{Si17d}).
We write
\begin{equation}
d = \begin{bmatrix} -1 \\ \tau_S \\ -\delta_S \end{bmatrix},
\nonumber
\end{equation}
so that $\tau_S$ and $\delta_S$ are the trace and determinant of
\begin{equation}
\tilde{M} = \begin{bmatrix} \tau_S & 1 \\ -\delta_S & 0 \end{bmatrix}.
\nonumber
\end{equation}
Thus there are five parameters ($\tau_L, \sigma_L, \delta_L, \tau_S, \delta_S \in \mathbb{R}$)
in addition to $\mu$.

Glendinning \cite{Gl17d} identified a Shilnikov homoclinic orbit
in a three-dimensional system of the form \eqref{eq:F2}.
This shows that the dynamical properties associated with Shilnikov homoclinic orbits, such as chaos,
can be generated in BEBs in three dimensions.

As an example, we use the parameter values
\begin{equation}
\begin{split}
\tau_L &= -0.5, \\
\sigma_L &= 4, \\
\delta_L &= 2,
\end{split}
\label{eq:param3d}
\end{equation}
so that the regular equilibrium is a saddle-focus (necessary for a Shilnikov homoclinic orbit).
We also fix $\delta_S = 1$ and consider various values of $\tau_S$.

Fig.~\ref{fig:ppBEB3d_framework} shows a phase portrait of \eqref{eq:normalForm} with $\tau_S = 0.275$ and $\mu = 1$
(with instead $\mu = -1$ the system does not appear to have a bounded attractor).
Here both $x^L$ and $x^S$ are admissible and unstable
(the BEB is of nonsmooth-fold type; in Theorem \ref{th:Feigin} we have $N_L = 1$, $N_S = 0$, and $c_1 < 0$).
Since the system is linear in $x_1 < 0$,
the stable and unstable manifolds of $x^L$ are linear as they emanate from $x^L$ (as shown in Fig.~\ref{fig:ppBEB3d_framework})
only becoming curved once they intersect $\Sigma$.
There is both an attracting limit cycle and an apparently chaotic attractor.
%(similar to that described in \cite{Si16c} for a continuous system).
Thus two attractors are created in the BEB at $\mu = 0$.
The creation of multiple attractors in a BEB
(for both Filippov systems and continuous systems)
does not appear to have been reported previously.

%%%%%%%%%%%%%%%%%%%%%%%%%%%%%%%%%%%%%%%%%%%%%%%%%%%%%%%%%%%%%%%%%%%%%%%%%%%%%%%%%%%%%%%%%%%%%%%%%%%%%%%%%%%%%%%%%%%%%%%%
\begin{figure}[t!]
\begin{center}
\setlength{\unitlength}{1cm}
\begin{picture}(8,5.3)
\put(0,0){\includegraphics[width=8cm]{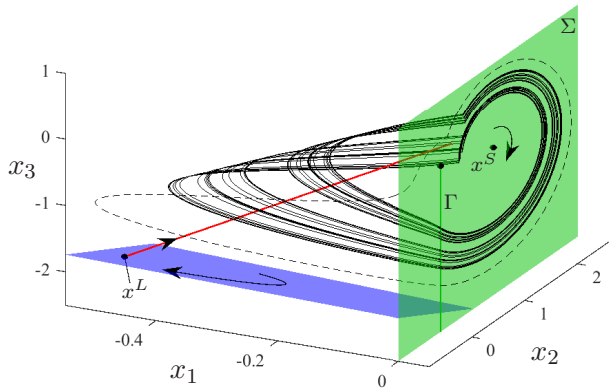}}
\put(2.15,.25){\small $x_1$}
\put(6.95,.5){\small $x_2$}
\put(0,2.97){\small $x_3$}
\put(1.51,1.32){\scriptsize $x^L$}
\put(6.15,3.01){\scriptsize $x^S$}
%\put(2.15,.25){\small $x$}
%\put(6.95,.5){\small $y$}
%\put(0,2.97){\small $z$}
%\put(1.51,1.32){\scriptsize ${\bf x}_L$}
%\put(6.19,3.13){\scriptsize ${\bf x}_S$}
\put(7.34,4.84){\scriptsize $\Sigma$}
\put(5.79,2.5){\scriptsize $\Gamma$}
\end{picture}
\caption{
A phase portrait of the normal form \eqref{eq:normalForm} in three dimensions.
Specifically the parameter values are \eqref{eq:param3d}, $\tau_S = 0.275$, $\delta_S = 1$, and $\mu = 1$.
A stable limit cycle is shown with a dashed curve;
an apparently chaotic attractor is shown with a solid curve.
The regular equilibrium $x^L$ is a saddle-focus;
the pseudo-equilibrium $x^S$ is an unstable focus.
\label{fig:ppBEB3d_framework}
} 
\end{center}
\end{figure}
%%%%%%%%%%%%%%%%%%%%%%%%%%%%%%%%%%%%%%%%%%%%%%%%%%%%%%%%%%%%%%%%%%%%%%%%%%%%%%%%%%%%%%%%%%%%%%%%%%%%%%%%%%%%%%%%%%%%%%%%

Next we construct a Poincar\'e map to better understand these attractors.
As discussed in \sect{properties}, sliding motion on $\Sigma$ turns into regular motion in $x_1 < 0$
along the line $x_1 = x_2 = 0$ with $x_3 < 0$.
We let $\Gamma$ denote this half-line
and $P$ denote the Poincar\'e map on $\Gamma$.
That is, for any $x_3 < 0$, we let $P(x_3)$ be such that
the forward orbit of $(0,0,x_3)$ next intersects $\Gamma$ at $\left( 0, 0, P(x_3) \right)$,
and let $P(x_3)$ be undefined if the orbit does not return to $\Gamma$.
The map $P$ captures all invariant sets of \eqref{eq:normalForm} that involve both sliding motion and regular motion.
The only other bounded invariant sets of \eqref{eq:normalForm} are the equilibria, $x^L$ and $x^S$,
because the sliding and regular dynamics are governed by linear systems.

Fig.~\ref{fig:slidingReturnMapBEB3d_framework} shows $P$
using the same parameter values as Fig.~\ref{fig:ppBEB3d_framework}.
The cobweb diagram indicates the apparently chaotic attractor.
The stable fixed point of $P$ (shaded blue)
corresponds to the attracting limit cycle.
The two unstable fixed points of $P$ (shaded red)
correspond to unstable limit cycles.
The map $P$ has oscillations accumulating at $x_3 \approx -2.21$
for which the forward orbit of $(0,0,x_3)$ converges to $x^L$,
for details see Glendinning \cite{Gl17d}.

%%%%%%%%%%%%%%%%%%%%%%%%%%%%%%%%%%%%%%%%%%%%%%%%%%%%%%%%%%%%%%%%%%%%%%%%%%%%%%%%%%%%%%%%%%%%%%%%%%%%%%%%%%%%%%%%%%%%%%%%
\begin{figure}[h!]
\begin{center}
\setlength{\unitlength}{1cm}
\begin{picture}(8,6)
\put(0,0){\includegraphics[width=8cm]{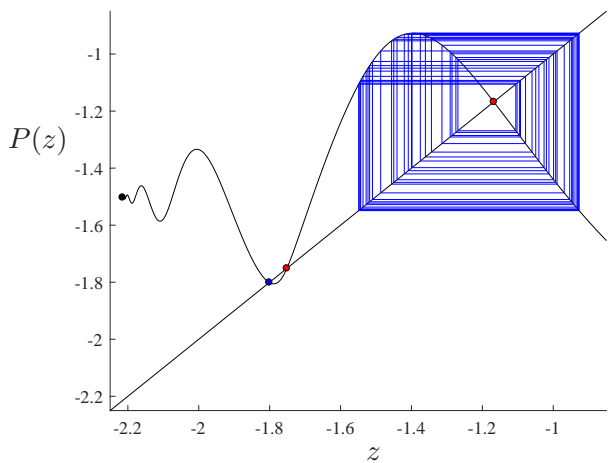}}
%\put(4.76,0){\small $x_3$}
%\put(0,4.15){\small $P(x_3)$}
\put(4.76,0){\small $z$}
\put(0,4.15){\small $P(z)$}
\end{picture}
\caption{
The Poincar\'e map $P$ for the parameter values of Fig.~\ref{fig:ppBEB3d_framework}.
\label{fig:slidingReturnMapBEB3d_framework}
} 
\end{center}
\end{figure}
%%%%%%%%%%%%%%%%%%%%%%%%%%%%%%%%%%%%%%%%%%%%%%%%%%%%%%%%%%%%%%%%%%%%%%%%%%%%%%%%%%%%%%%%%%%%%%%%%%%%%%%%%%%%%%%%%%%%%%%%

Fig.~\ref{fig:bifDiagBEB3d_framework} shows a bifurcation diagram
of \eqref{eq:normalForm} produced by varying the value of $\tau_S$.
More specifically, Fig.~\ref{fig:bifDiagBEB3d_framework} indicates
the long-term behaviour of $P$ with the initial value $x_3 = -1$.
By using this particular initial value, we observe bifurcations 
leading to the apparently chaotic attractor of Fig.~\ref{fig:ppBEB3d_framework}.
In particular we see period-doubling cascades and windows of periodicity
(typical for smooth non-invertible one-dimensional maps)
as well as adding-sliding bifurcations \cite{DiBu08} at $\tau_S \approx 0.142$ and $\tau_S \approx 0.163$
where the attractor includes the value $x_3 = 0$.

%%%%%%%%%%%%%%%%%%%%%%%%%%%%%%%%%%%%%%%%%%%%%%%%%%%%%%%%%%%%%%%%%%%%%%%%%%%%%%%%%%%%%%%%%%%%%%%%%%%%%%%%%%%%%%%%%%%%%%%%
\begin{figure}[h!]
\begin{center}
\setlength{\unitlength}{1cm}
\begin{picture}(8,6)
\put(0,0){\includegraphics[width=8cm]{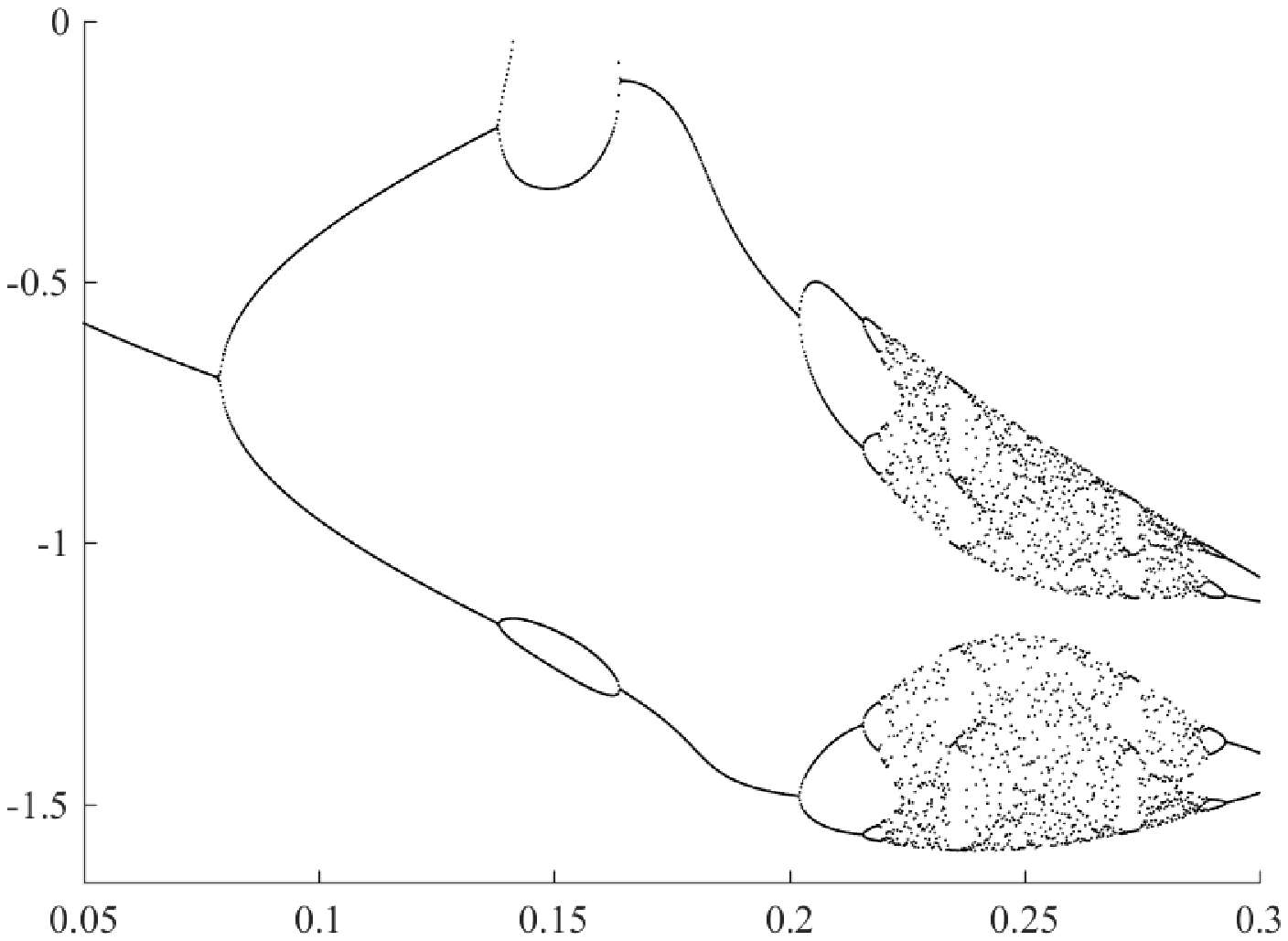}}
\put(4.14,0){\small $\tau_S$}
\put(0,3.43){\small $x_3$}
\end{picture}
\caption{
A bifurcation diagram of the normal form \eqref{eq:normalForm}
in three dimensions using the parameter values \eqref{eq:param3d}, $\delta_S = 1$, and $\mu = 1$.
The vertical axis represents the $x_3$-value on $\Gamma$ (where $x_1 = x_2 = 0$).
\label{fig:bifDiagBEB3d_framework}
} 
\end{center}
\end{figure}
%%%%%%%%%%%%%%%%%%%%%%%%%%%%%%%%%%%%%%%%%%%%%%%%%%%%%%%%%%%%%%%%%%%%%%%%%%%%%%%%%%%%%%%%%%%%%%%%%%%%%%%%%%%%%%%%%%%%%%%%

%Finally, Fig.~\ref{eq:ppBEB3d_dangerous_framework} shows a phase portrait using
Finally we consider the parameter values
\begin{equation}
\begin{split}
\tau_L &= -0.3, \\
\sigma_L &= 0.4, \\
\delta_L &= -0.1, \\
\tau_S &= -0.2, \\
\delta_S &= 1.
\end{split}
\label{eq:param3d2}
\end{equation}
For these values all eigenvalues of $C$ and $\tilde{M}$ have negative real part.
Thus, when admissible, the equilibria $x^L$ and $x^S$ are asymptotically stable.
Yet, as shown in Fig.~\ref{fig:ppBEB3d_dangerous_framework}, with $\mu = 0$ the origin is unstable.
This is because the system has an invariant cone involving a sector of $\Sigma$
over which orbits head outwards, and this dominates the inwards motion in $x_1 < 0$ to create a net outwards motion.
This counter-intuitive result is not possible in less than three dimensions
and is well-known for continuous systems \cite{CaFr06}.

%%%%%%%%%%%%%%%%%%%%%%%%%%%%%%%%%%%%%%%%%%%%%%%%%%%%%%%%%%%%%%%%%%%%%%%%%%%%%%%%%%%%%%%%%%%%%%%%%%%%%%%%%%%%%%%%%%%%%%%%
\begin{figure}[h!]
\begin{center}
\setlength{\unitlength}{1cm}
\begin{picture}(8,6)
\put(0,0){\includegraphics[width=8cm]{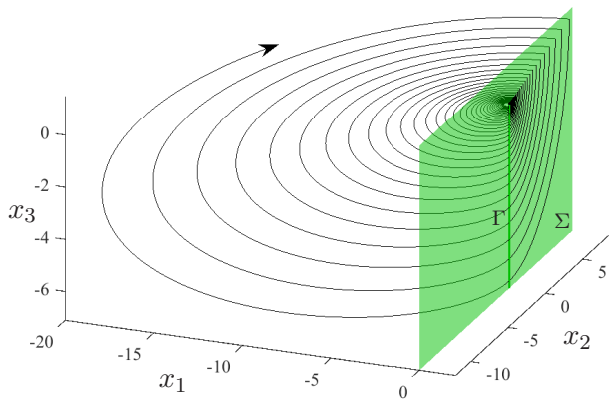}}
\put(2,.25){\small $x_1$}
\put(7.38,.82){\small $x_2$}
\put(0,2.5){\small $x_3$}
\put(7.26,2.36){\scriptsize $\Sigma$}
\put(6.43,2.4){\scriptsize $\Gamma$}
\end{picture}
\caption{
A phase portrait of \eqref{eq:normalForm} in three dimensions with
\eqref{eq:param3d2} and $\mu = 0$.
\label{fig:ppBEB3d_dangerous_framework}
} 
\end{center}
\end{figure}
%%%%%%%%%%%%%%%%%%%%%%%%%%%%%%%%%%%%%%%%%%%%%%%%%%%%%%%%%%%%%%%%%%%%%%%%%%%%%%%%%%%%%%%%%%%%%%%%%%%%%%%%%%%%%%%%%%%%%%%%

%===============================================================================
\section{Discussion}
\label{sec:conc}

The $n$-dimensional normal form \eqref{eq:normalForm} has $2 n - 1$ parameters
(not counting the BEB parameter $\mu \in \mathbb{R}$, or $d_1 = \pm 1$).
The parameters are $a_1,\ldots,a_n \in \mathbb{R}$, which appear in the companion matrix $C$,
and $d_2,\ldots,d_n \in \mathbb{R}$, which appear in the vector $d$.
For a given BEB in a system of the form \eqref{eq:F2},
the values of these parameters can be determined by explicitly applying the coordinate transformation
of Theorem \ref{th:normalForm},
or, more simply, they can be determined from the eigenvalues of $A = D F^L(\b0;0)$ and $M = D F^S(\b0;0)$, as follows.
First set $d_1 = {\rm sgn}(c_1)$.
%Assuming $A$ has no eigenvector orthogonal to $e_1$
%(so that the transformation can be achieved),
The eigenvalues of $A$ uniquely determine $a_1,\ldots,a_n$ because
these numbers are the coefficients of the characteristic polynomial of $A$
(we know that $A$ and $C$ are similar).
The eigenvalues of $M$ uniquely determine $d_2,\ldots,d_n$ in a similar fashion.
Specifically, $M$ has a zero eigenvalue,
and the remaining $n-1$ eigenvalues %(some of which could be zero)
are those of $\tilde{M}$, whose characteristic polynomial
has coefficients $\frac{d_2}{d_1},\ldots,\frac{d_n}{d_1}$, see \eqref{eq:tildeM}.
Moreover, we conclude that the eigenvalues of $A$ and $M$
fully determine all structurally stable features of the dynamics local to the BEB,
assuming non-degeneracy conditions are satisfied.

The normal form is intended to provide a foundation by which BEBs can be analysed in systems of any number of dimensions.
Generic BEBs in two dimensions have been completely classified \cite{HoHo16}.
This paper reveals new complexities for BEBs in three dimensions, such as the creation of multiple attractors.
The results of \sect{3d} provide numerical evidence for the creation of a chaotic attractor in a BEB,
but the given example has no equilibrium on the other side of the BEB.
It remains to be determined if a stable equilibrium
can bifurcate to a chaotic attractor in a BEB of a Filippov system
(this has been found numerically in continuous systems \cite{Si16c}).

In \sect{3d} the three-dimensional dynamics was captured with a Poincar\'e map $P$.
For smooth three-dimensional systems of ODEs, Poincar\'e maps are two-dimensional.
Here $P$ is one-dimensional because orbits become constrained to
the codimension-one discontinuity surface $\Sigma$.
Yet, as in Fig.~\ref{fig:slidingReturnMapBEB3d_framework}, $P$ may be non-invertible.
This is because the contraction of orbits onto $\Sigma$ is a non-invertible process,
and permits the dynamics to be chaotic.
If we were to smooth the system by mollifying the switching condition,
the dynamics would be captured by a two-dimensional Poincar\'e map.
This map would be invertible but well-approximated by the one-dimensional non-invertible map $P$.
The approximation of three-dimensional dynamics with a one-dimensional map
occurs also for the dynamics associated with a Shilnikov homoclinic orbit \cite{GlSp84}
and an orbit homoclinic to a limit cycle \cite{GaSi72}, for example.

\end{document}